\documentclass[11pt]{amsart}

\usepackage[left=3.4cm,right=3.4cm,top=2.2cm,bottom=2.2cm]{geometry}

\usepackage{amsmath,amssymb,mathrsfs,stmaryrd}
\usepackage{amsthm}
\usepackage{mathtools,bm}

\usepackage[LGR,T1]{fontenc} 
\usepackage[utf8]{inputenc} %
\usepackage{lmodern}
\usepackage[english]{babel}
\usepackage{cancel}
\usepackage{dsfont}
\usepackage[colorlinks,linkcolor=NavyBlue,citecolor=Gray,urlcolor=Periwinkle]
{hyperref}
\usepackage[nameinlink,capitalize]{cleveref}
\usepackage[dvipsnames]{xcolor}
\usepackage[shortlabels]{enumitem}
\usepackage[babel=true,kerning=true]{microtype}
\usepackage[abbrev,msc-links,alphabetic]{amsrefs}
\usepackage{lscape}
\usepackage{afterpage}

\usepackage[textsize=tiny]{todonotes}

\theoremstyle{plain}

\newcommand{\vep}{\varepsilon}
\newcommand{\ome}{\omega}

\newcommand{\Cb}{\ensuremath{\mathcal C_b}}

\newcommand{\be}{\begin{equation}}
\newcommand{\ee}{\end{equation}}
\newcommand{\bea}{\begin{eqnarray}}
\newcommand{\eea}{\end{eqnarray}}

\newcommand{\beaa}{\begin{eqnarray*}}
\newcommand{\eeaa}{\end{eqnarray*}}

\newcommand{\eps}{\varepsilon}

\newcommand{\MF}{\cff}

\newcommand{\MC}{\mathscr{C}}
\newcommand{\ML}{\mathrm{Lin}}

\newcommand{\MA}{\mathcal{A}}

\newcommand{\MG}{\mathcal{G}}
\newcommand{\MS}{\mathcal{S}}
\newcommand{\MO}{\ooo}

\newcommand{\Bx}{\mathbf{x}}
\newcommand{\BW}{\mathbf{W}} 
\newcommand{\BD}{\mathbf{D}}

\newcommand{\MV}{\mathcal{V}}

\newcommand{\tY}{\tilde{Y}}

\newcommand{\tZ}{\tilde{Z}}

\newcommand{\tMV}{\tilde{\MV}}

\newcommand{\bY}{\bar{Y}}

\newcommand{\bMV}{\bar{\MV}}
\newcommand{\bBX}{\bar{\mathbf{X}}}
\newcommand{\bt}{\bar{t}}
\newcommand{\by}{\bar{y}}

\newcommand{\E}{\mathbb{E}}

\newcommand{\R}{\mathbb{R}}
\newcommand{\N}{\mathbb{N}}

\newcommand{\X}{\mathbb{X}}

\newcommand{\BX}{\mathbf{X}}
\newcommand{\By}{\mathbf{y}}

\newcommand{\op}{\mathcal{P}}

\newcommand{\bx}{\mathbf{x}}

\newtheorem{thm}{Theorem}[section]
\newtheorem{theorem}[thm]{Theorem}
\newtheorem{assumption}[thm]{Assumption}
\newtheorem{lem}[thm]{Lemma}
\newtheorem{coro}[thm]{Corollary}
\newtheorem{corollary}[thm]{Corollary}
\newtheorem{rem}[thm]{Remark}
\newtheorem{remark}[thm]{Remark}
\newtheorem{prop}[thm]{Proposition}
\newtheorem{proposition}[thm]{Proposition}
\newtheorem{exam}[thm]{Example}
\newtheorem{defi}[thm]{Definition}
\newtheorem{definition}[thm]{Definition}

\newcommand{\cdummy}{\cdot}
\newcommand{\tmem}[1]{{\em #1\/}}

\numberwithin{equation}{section}
\newcommand{\assign}{:=}

\newcommand{\tmop}[1]{\ensuremath{\operatorname{#1}}}

\newcommand{\bk}[1]{\llbracket#1\rrbracket}
\newcommand{\tand}{\quad\textrm{and} \quad}

\newcommand{\ctrl}{\eta}

\newcommand{\bctrl}{\bar{\eta}}

\newcommand{\caa}{\mathcal A}

\newcommand{\aaa}{\mathfrak A}
\newcommand{\cff}{{\mathfrak F}}  
\newcommand{\uuu}{\mathfrak U} 
\newcommand{\bbb}{\mathfrak B} 
\newcommand{\ooo}{\mathfrak O} 
\newcommand{\mmm}{\mathfrak M} 

\newcommand{\cpp}{{\mathcal P}}

\newcommand{\XX}{\ensuremath{\mathbb{X}}}
\newcommand{\Rd}{{\R^d}}
\renewcommand{\P}{\ensuremath{{\mathbb P}}}
\newcommand{\D}{\ensuremath{{\mathbf D}}}

\newcommand{\lip}{\ensuremath{{\mathrm{Lip}}}}

\DeclareMathOperator*{\esssup}{ess\,sup}

\newcommand{\C}{{\mathcal{C}}}
\newcommand{\CC}{{\mathscr{C}}}

\newcommand{\Ca}{{\mathscr{C}}^{\alpha}}
\newcommand{\Coa}{{\mathscr{C}}^{0,\alpha}}
\newcommand{\BCa}{{\mathfrak{C}}^{\alpha}}
\newcommand{\BCoa}{{\mathfrak{C}}^{0,\alpha}}

\usepackage{graphicx}

\makeatletter
\newcommand*\bigcdot{{\mathpalette\bigcdot@{.5}}}
\newcommand*\bigcdot@[2]{\mathbin{\vcenter{\hbox{\scalebox{#2}{$\m@th#1\bullet$}}}}}
\makeatother

\DeclareFontEncoding{FMS}{}{}
\DeclareFontSubstitution{FMS}{futm}{m}{n}
\DeclareFontEncoding{FMX}{}{}
\DeclareFontSubstitution{FMX}{futm}{m}{n}
\DeclareSymbolFont{fouriersymbols}{FMS}{futm}{m}{n}
\DeclareSymbolFont{fourierlargesymbols}{FMX}{futm}{m}{n}
\DeclareMathDelimiter{\vvert}{\mathord}{fouriersymbols}{152}{fourierlargesymbols}{147}
\DeclarePairedDelimiter{\nn}{\vvert}{\vvert}

\begin{document}

\title[Controlled RSDEs, pathwise stochastic control, DPPs]{Controlled rough SDEs, pathwise stochastic control and dynamic programming principles}

\author{Peter K. Friz}
\address{TU Berlin and WIAS Berlin}
\email{friz@math.tu-berlin.de}

\author{Khoa L\^e}
\address{University of Leeds}
\email{k.le@leeds.ac.uk}

\author{Huilin Zhang}
\address{Shandong University and Humboldt University}
\email{huilinzhang@sdu.edu.cn}


\begin{abstract}
 
We study stochastic optimal control of rough stochastic differential equations (RSDEs). This is in the spirit of the pathwise control problem (Lions--Souganidis 1998, Buckdahn--Ma 2007; also Davis--Burstein 1992), 
with renewed interest and recent works drawing motivation from filtering, SPDEs, and reinforcement learning. 
Results include regularity of {\em rough} value functions, validity of a {\em rough} dynamic programming principle and new {\em rough} stability results for HJB equations, removing excessive regularity demands previously imposed by flow transformation methods. Measurable selection is used to relate RSDEs to ``doubly stochastic'' SDEs under conditioning.  In contrast to previous works, Brownian statistics for the to-be-conditioned-on noise are not required, aligned with the ``pathwise'' intuition that these should not matter upon conditioning. Depending on the chosen class of admissible controls, the involved processes may also be anticipating. The resulting stochastic value functions coincide in great generality for different classes of controls. RSDE theory offers a powerful and unified perspective on this problem class. 

\bigskip

\end{abstract}

\maketitle

\tableofcontents
\section{Introduction} 
\label{sec.introduction}
Consider a {partially} controlled multidimensional diffusion $Y = Y^{\theta}$ with It\^o dynamics
\be dY_t = b (t,Y_t ; \theta_t) dt + \sigma (t,Y_t ; \theta_t) d B_t + f (t,Y_t) d W_t . \label{equ:11}
\ee
Here $B$ and (for the moment) $W$ are independent Brownian motions, $\theta$ is taken in some class of ``admissible'' controls $\Theta$. {\em Pathwise stochastic control} 
is concerned with minimizing some expected cost, conditionally on $W$. Specifically, 
\be  V^{\Theta} (s, y, \omega) := \mathrm{essinf}_{\theta  \in \Theta}
    \mathbb{E}^{s, y}  \left( g (Y _T)  + \int_s^T \ell (t, Y_t, \theta_r)dr \Big| \mathfrak{F}_T^W  \right)  \label{equ:12}
 \ee
defines a random field, which is related to (see \cite{LS98,LS98b,BM07}) non-linear stochastic partial differential equations (SPDEs) of Hamilton--Jacobi--Bellman (HJB) type, of the form 
 \be \label{equ:sHJB}
        - d_t v = H (y, t, D v, D^2 v) dt  + (f (t,y) \cdot Dv) \circ d W_t, \qquad v(T,\cdot) \equiv g.
 \ee
 Classical It\^o theory is not well equipped to treat this SPDE. Taking a {\em pathwise} view, with $W(\omega)$ replaced by a generic continuous path $X$, it was proposed in \cite{LS98b} to understand such value functions as ``stochastic viscosity solution'' of \eqref{equ:sHJB}: their analysis is based on a deterministic path $X$, the problem only becomes ``stochastic'' at the very end upon {\em randomisation}, 
  \[ X \rightsquigarrow W(\omega) ; \]
 with no Brownian assumptions necessary, it suffices that $W$ has continuous realisations. The same logic applies to multidimensional $W$ and general vector fields $f$, upon taking a {\em rough-pathwise} perspective \cite{CFO11}. (The notion of pathwise stochastic control actually goes back at least to Davis and Burstein \cite{DB92}, cf.  \cite{DFG17, AC20} for a discussion.) 
 
 A direct stochastic approach to \eqref{equ:sHJB} and the associated control problem was pursued in \cite{BM07}, assuming Brownian statistics for $(B,W)$, essentially relying on stochastic flow transformations such as to ``transform away'' the $dW$-term in \eqref{equ:11} and \eqref{equ:sHJB}, respectively. 
Importantly, \cite{BM07} offers profound results to the control problem itself. A key subtlety lies in the ambiguity what one means by ``admissible''. While classical stochastic control theory suggests to consider controls adapted to the information generated by $(B,W)$ (for which the authors show that one may not even have a ``minimizing'' sequence), the particular form of the conditioning in \eqref{equ:12} suggests to work with the augmented filtration that contains all the information of $W$ until terminal time $T$. Such controls however, if employed in \eqref{equ:11} would lead to considerations of  anticipating stochastic differential equations, avoided in \cite{BM07} through stochastic flow transformations. Surprisingly perhaps, the resulting (random) value functions coincide and stochastic dynamic programming principles are seen to hold, though formulations of these results are somewhat indirect due to the said transformation approach. 

If one aims to connect \eqref{equ:12} with classical stochastic control theory, one can add constraints to enforce suitable adaptedness of the controls, as was carried out (in case of $\sigma \equiv 0$) in \cite{DFG17}, building on ideas of \cite{DB92, rogers2007pathwise}, see also \cite{AC20, BCO23, CHT24} for recent contributions to pathwise stochastic control in the context of filtering, SPDEs, and reinforcement learning, respectively.  

Classically, the restriction to adapted controls is tied to the interpretation of the time variable $t$ in \eqref{equ:11} and \eqref{equ:12} as {\em physical time} and the idea that one does not know the future. Numerous works in machine learning have led to other interpretations: for instance, a controlled evolution on $[0,T]$ has seen fruitful interpretations as continuous limit of a deep learning networks in the infinite layer limit, in which case  the ``future'' of the noise (e.g. induced by initialisation of deep neural network weights)  is fully available.  (It is not the purpose of this work to offer an interpretation of \eqref{equ:12} in this direction, but see \cite{bayer2023stability, gassiat2024gradient} for works in this spirit.) 

The purpose of this work is to revisit \eqref{equ:11} and \eqref{equ:12} through the lenses of {\em rough stochastic differential equations} \cite{FHL21}. In essence, this allows us to replace $W$ in \eqref{equ:11}, \eqref{equ:12} by a (purely deterministic) $\alpha$-H\"older rough path $\BX=(X,\mathbb{X})$, while maintaining Brownian statistics for $B$, in the form of a $(\mathfrak{F}_t)$-Brownian motion on some probability space $(\Omega, (\mathfrak{F}_t)_{t\ge0},
\mathfrak{F} , \mathbb{P})$. Consider the {\em rough} SDE (following \cite{FHL21}, with a quick review in \cref{sec:RSDE}) 
\be dY^\BX_t = b (t,Y^\BX_t ; \eta_t) dt + \sigma (t,Y^\BX_t ; \eta_t) d B_t + (f,f') (t,Y^\BX_t) d \BX_t , \label{equ:11r}
   \ee
with adapted controls $\eta \in \mathcal{A}$ (cf. Section \ref{def:ad-cont})  and 
\be
\mathcal{V} (s, y ; \mathbf{X} ) := \inf_{\eta  \in \mathcal{A} }
   \mathbb{E}^{s, y}  \left( g (Y^{\BX}_T)  + \int_s^T \ell (t, Y^{\BX}_t, \eta_r)dr   \right) , \label{equ:12r}
 \ee
which we call {\em rough} value function. Our approach offers a variety of advantages compared to earlier works.
\begin{itemize}
\item No more ambiguity on the notion of admissible control: with $\BX$ being  deterministic, there is only one filtration $(\mathfrak{F}_t)$.
\item With \eqref{equ:11r}, direct (rough)path-wise meaning is given to \eqref{equ:11}, i.e. without imposing (ultimately irrelevant) Brownian statistics for $W$.
\item Dynamic programming principle (DPP) for \eqref{equ:12r} is formulated with ``$\inf$'' rather than ``$\mathrm{essinf}$'', the existence of minimizing sequences is trivially guaranteed. 
\item Regularity of the value function comes directly from \eqref{equ:12r}, bypassing any need to construct careful modifications of \eqref{equ:12}.
\item Reduction of regularity demands \cite{CFO11}, e.g. when $\alpha=(\tfrac{1}{2})^-$, from $f \in \C_b^4$ to $\C_b^{2+}$, in rough path stability of HJB equations.
\end{itemize} 
All this is achieved in  \cref{DPP_for_RSDEs} and  \cref{sec:HJB}. Let us note here that $(f,f')$ is a deterministic controlled (in sense of Gubinelli) vector field which allows for ``rough'' time dependence in $f$. The integral meaning of the last term in \eqref{equ:11r} is then, in the sense of u.c.p. convergence,
$$
\int (f,f') (r,Y^\BX_r) d \BX_r \sim \sum \Big( f (s, Y^\BX_s) (X_t -X_s) + ((Df_s) f_s +f'_s) (Y_s^\BX) (\mathbb{X}_{s,t}) \Big),
$$
leaving additional reminders on stochastic rough integrals to \cref{sec:RSDE}. At first reading the reader may consider $f' \equiv 0$ which accommodates the case of autonomous $f$, and also Young-complementary $t$-regularity, neither of which holds if $f$ were to depend on $\mathrm{Law}(Y_t)$: the generality of our setup is precisely rooted in preparing the grounds for mean-field controlled rough stochastic differential equations, which (as also suggested in \cite{carmona2014master}) 
ultimately may offer new tools to the analysis of mean field games conditionally on common noise.
\medskip

\noindent {\em Randomisation: } To some extent, this paper could have ended here, at least if one is prepared to regard a rough path as perfect model for a noise that one wants to treat as ``frozen'' analytic object.  That said, the remainder of this work is devoted to the randomisation of the afore-mentioned results,
$ \mathbf{X} \rightsquigarrow \mathbf{W} (\omega).$
This is natural, on the one hand, since many examples of rough paths actually do come from stochastic processes, enhanced with iterated integrals or L\'evy's area, e.g. \cite{FV10}. On the other hand, this allows to reconnect with previous works, notably \cite{BM07}, leading to a number of improvements, including
\begin{itemize}
\item flexible meaning to \eqref{equ:11}, without Brownian assumptions on $W$ (it is immediate to treat e.g. fractional Brownian motion in the regime $H>1/3$),
\item reduction of regularity demands \cite{BM07}, again from $f \in \C_b^4$ to $\C_b^{2+}$,
\item modulus of continuity for the random field \eqref{equ:12r}, left as an open problem in \cite{BM07},
\item unified treatment of \eqref{equ:11} for different classes of admissible controls (not depending on $W$, adapted, $W$-anticipating),
\item general understanding why all lead to  the same stochastic value function, validity of stochastic DPP .
\end{itemize} 
We start with the {\em 1st randomisation} of the value function,
\[ \bar{\mathcal{V}} (s, y ; \mathbf{} \omega) \assign \mathcal{V} (s, y ;
   \mathbf{} \mathbf{W} (\omega)) =\mathcal{V} (s, y ; \mathbf{}
   \mathbf{X}) |_{\mathbf{X} = \mathbf{W} (\omega)}.   \]
The {\em 2nd randomisation} concerns the RSDE itself. To this end we move to a product space that supports a Brownian motion $B=B(\omega')$ and a random rough path  $\mathbf{W} (\omega'') \perp B (\omega')$, writing $\omega = (\omega',
\omega'')$. Given a family of $\mathfrak{F'}$-adapted controls $\eta_t (\omega' ,\BX)$, indexed by $\BX$, we can solve, for each $\BX$, the 
 rough SDE $Y = Y^{\eta, \mathbf{X}} (\omega')$ on
$\Omega'$, started from some fixed $Y_0 = \xi$, 
\[ d Y_t = b (t, Y_t;  \eta_t (\omega' , \BX)) d t + \sigma (t, Y_t ; \eta_t (\omega' ,
   \BX)) d B_t (\omega')  + f (t,Y_t) d \BX_t  .
    \]
Write $\bar{\eta}_t (\omega) = \eta_t (\omega' , \mathbf{W} (\omega''))$ and then $\bar{Y}^{\bar{\eta}} (\omega) = Y^{\bar{\eta},
\mathbf{W} (\omega'')} (\omega')$ for the respective randomisation of $\eta$ and $Y$. 
By $( \cff'_t)$-adaptedness of controls, and under minimal measurability assumptions (on $\eta$ and $\mathbf{W}$; details left to Section \ref{sec:pathwise-control}) 
we can expect $\bar{\eta}_t$ to be  $\in \cff'_t
\vee \cff''_T$, hence partially anticipating.
Thanks to 
a number of new measurable selection results for RSDEs (interesting in their own right, \cref{sec:measurable_prelim}) this defines
measurable process, with
\[ \bar{Y}^{\bar{\eta}}_t (\omega', \omega'') \in \cff'_t \vee
   \cff''_T . \]
In this sense, every class $\mathcal{A}^{\circ}$ of (suitably jointly measurable) RSDE controls $\eta = \eta (\omega', \BX)$, 
induces a class of randomised controls $\bar{\mathcal{A}}^{\circ}$, defined by $\bar{\eta}_t (\omega) = \eta_t (\omega' , \mathbf{W} (\omega''))$, which gives rise to a random value function
\[ \mathcal{V}^\circ (s, y ; \mathbf{} \omega) \assign
    {\tmop{essinf}_{\eta  \in \mathcal{A}^\circ}} \mathbb{E}^{s,
   y} (g (\bar{Y}^{\bar{\eta}}_T)   {| \cff_T^{\mathbf{W}}
    }) =  {\tmop{essinf}_{\bar{\eta}  \in
   \bar{\mathcal{A}}^\circ}} \mathbb{E}^{s, y} (g (\bar{Y}^{\bar{\eta}}_T)
     {| \cff_T^{\mathbf{W}}  }). \]
  
Unlike the situation of \eqref{equ:12r} where there was no ambiguity about the class of $( \cff'_t)$-admissible controls, we here have a choice since $\bar{\mathcal{A}}^{\circ}$ manifestly depends on how the controls in $\mathcal{A}^{\circ}$ are allowed to depend on the rough path $\BX \in \mathscr{C}_T$. 
Consider the two extreme choices:

\medskip

\noindent (1) Let  $\mathcal{A}^1$ be the ``minimal'' choice of controls $\eta = \eta (\omega',\BX)$ which in fact are not allowed to depend on $\BX$ (this will be formalised in Section 5 via the trivial $\sigma$-field on rough path space). 

\medskip

\noindent (2) Let $\mathcal{A}^2$ be the ``maximal'' choice that allows for any (measurable) dependence in $\BX = \{ \BX_t : 0 \le t \le T \}$. In this case, controls in $\bar{\mathcal{A}}^2$ are only adapted to the partially augmented filtration $\{\cff'_t \vee \cff''_T: 0 \le t \le T\}$, thus partially anticipating.
\medskip

Of course, there are many other choices $\mathcal{A}^{\circ}$ (with $\mathcal{A}^{1}\subset \mathcal{A}^{\circ}\subset \mathcal{A}^{2}$), 
such as $\mathcal{A}^{\mathrm{causal}}$, defined as those controls in $\mathcal{A}^2$ which are causal in $\BX$ (see equation \ref{equ:Acausal} below), in which case $\bar{\mathcal{A}}^{\mathrm{causal}}$ consists of jointly adapted controls, with $\bar \eta_t \in \{\cff'_t \vee \cff''_t\}$. But since $\mathcal{V}^\circ$ is monotone w.r.t. $\mathcal{A}^\circ$, it will indeed be enough to treat the extreme cases.

Introducing the respective random value functions, 
\[ \mathcal{V}^i (s, y ; \mathbf{} \omega) \assign
    {\tmop{essinf}_{\bar{\eta}  \in
   \bar{\mathcal{A}}^i}} \mathbb{E}^{s, y} (g (\bar{Y}^{\bar{\eta}}_T)
     {| \cff_T^{\mathbf{W}}  }),\quad i=1,2, \]
our main result in \cref{sec:pathwise-control}, valid in the full generality of a (level $2$, H\"older) random rough path $\mathbf{W}(.)$,  can be summarised by saying that the randomised rough value function $ \bar{\mathcal{V}}$ is a continuous modification of both $\mathcal{V}^1$ and $\mathcal{V}^2$ with a.s. regularity inherited from the regularity of the rough value function $\mathcal{V}$ (\cref{thm:RoughValueReg}). In particular, 
for every $(s,y) \in [0,T] \times \R^{d_Y}$, with probability one,
 $$
\bar{\mathcal{V}} (s, y ;  \omega) = \mathcal{V}^1 (s,
  y ;  \omega) =\mathcal{V}^2 (s, y ;  \omega).$$
We furthermore see that a stochastic DPP holds for the random fields $\mathcal{V}^1$ and $\mathcal{V}^2$, essentially as consequence of the deterministic dynamic programming principle  for the rough value function $\mathcal{V}$. 

Our final \cref{sec:BMcase} is devoted to specialize to the case when $B,W$ has jointly Brownian dynamics. More specifically, assume that $\BW$ is the (It\^o) Brownian rough path,
and $\bar{\eta} = \bar{\eta}_t (\omega) \in \bar{\mathcal{A}}^1$ then
$\bar{Y}^{\bar{\eta}} (\omega)$ solves the It\^o SDE 
\[ d Y = b (t,Y_t; \eta_t (\omega')) d t + \sigma (t,Y_t; \eta_t (\ome')) d B (\omega') + f
   (t,Y_t) d W (\omega'') \]
and this remains true for $\mathbf{X}$-causal controls (
i.e. $\eta_t (\ome' , \mathbf{X}) = \eta_t (\ome' , \mathbf{X}_{. \wedge t})$) in
which case $\bar{\eta}$ is \ \{$\cff'_t \vee \cff''_t$\}-adapted, situation(s) considered in \cite{BM07}. For $\bar{\eta} = \bar{\eta}_t (\omega) \in \bar{\mathcal{A}}^2$,
$\bar{\eta} $ is not $\cff'_t \vee \cff''_t$ adapted, so that
$\bar{Y}^{\bar{\eta}}$ can be seen as solution to some anticipating SDE: this is provided {\em en passant} by our randomised RSDE approach, with no reliance on any results
from anticipating stochastic calculus or stochastic flow transformations (as employed in \cite{BM07}).

Randomised RSDEs thus also provide a unification, that deals simultaneously
with {partially anticipating coefficients} and non-Brownian noise $W$ (independent of Brownian noise 
$B$). 

\medskip

\section{Preliminaries and Notation} \label{sec:notation}
\subsection{Generalities}  A filtered probability space is denoted by $(\Omega, \cff, (\cff_t)_{t \geq 0}, \P)$ and said to satisfy the {\em usual conditions} if $\cff_t = \bigcap_{s > t} \cff_s \  \text{for all } t \geq 0$ and each \(\cff_t\) contains all \( \mathbb{P} \)-null sets of \( \cff \). Given generic measure spaces $(M_i,\mmm_i), i =1,2$ we say $f:M_1 \to M_2$ is $\mmm_1 / \mmm_2$-measurable if $f^{-1} (\mmm_2) \subset \mmm_1$. If $U$ is a topological space, we equip it with its Borel sets, $\uuu = \bbb (U)$, unless otherwise stated; if furthermore $U$ is Polish, we call $(U,\uuu) = (U,\bbb (U))$ a {\em Polish measure space}.
We also write $\bbb_T := \bbb ([0,T])$ and $\bbb^d := \bbb (\R^d)$.

Throughout the paper, \(T>0\) is a fixed time horizon and \(\Delta_T\) denotes the simplex \(\{(s,t)\in[0,T]^2:s\le t\}\). 
For a path \((Z_t)_{t\in[0,T]}\) and a two-parameter path \((A_{s,t})_{(s,t)\in \Delta_T}\), we denote \(\delta Z_{s,t}\) for the increment \(Z_t-Z_s\) for each \(s,t\) and for each \(\beta>0\) we set
\begin{align*}
   |A|_\beta=\sup_{(s,t)\in \Delta_T}\frac{|A_{s,t}|}{(t-s)^\beta}.
\end{align*}
\subsection{Rough paths}  Let $\alpha \in (\tfrac{1}{3}, \tfrac{1}{2}]$. Following \cite[Sec. 2]{FH20} the H\"older rough path space
\begin{equation}
	\label{def.Ca}
\mathscr{C}_T^{\alpha} := \mathscr{C}^{\alpha}([0,T]; \R^d), \\[0.3em]
\end{equation}
equipped with the (``inhomogeneous'') metric
\begin{equation}
	\label{def.rho_metric}
		\rho_{\alpha}(\BX,\bar\BX)=|\delta X- \delta\bar X|_\alpha+|\XX-\bar\XX|_{2\alpha}, 
\end{equation}
is a complete but non-separable metric space.
The homogeneous rough path norm is given by $$ \nn{\BX}_\alpha= |\delta X|_\alpha \vee \sqrt{|\XX|_{2 \alpha}}. $$  
Several closed subspaces of $\mathscr{C}_T^{\alpha}$ are of interest,
\begin{align*}
  \mathscr{C}_T^{0,\alpha} 
    &:= \mathscr{C}^{0,\alpha}([0,T]; \R^d), \\[0.3em]
  \mathscr{C}_{g,T}^{\alpha} 
    &:= \mathscr{C}_g^{\alpha}([0,T]; \R^d), \\[0.3em]
  \mathscr{C}_{T}^{0,\alpha} 
    &:= \mathscr{C}_g^{0, \alpha}([0,T]; \R^d), 
\end{align*}
where 
$\mathscr{C}_{g,T}^{\alpha}$ is the (non-separable) space of weakly geometric $\alpha$-H\"older rough paths, 
whereas $\mathscr{C}_{g,T}^{0,\alpha}$ is the (Polish) space of $\alpha$-H\"older geometric rough paths obtained as the 
$\rho_\alpha$-closure of canonically lifted smooth $\R^d$-valued paths. A bit less well-known perhaps, though discussed in detail in \cite[Ex. 2.8, 2.12]{FH20},  is the (Polish) space
$
\mathscr{C}_T^{0,\alpha}, 
$ 
strictly bigger than $\mathscr{C}^{0,\alpha}_{g;T}$, obtained as closure of (possibly non-geometric) smooth rough paths.\footnote{Warning: In modern terminology, motivated by the notion of smooth model in regularity structures, smooth rough paths are simply smooth paths with values in  $\mathbb{R}^d \oplus \mathbb{R}^{d\times d}$. For example, when $d=1$, the map $t \mapsto (0;t)$ is an example of a  smooth (non-geometric, pure second level) rough path.}

The notion of $\BX$-causality will be applied to (adapted) stochastic processes that depend additionally on $\mathbf{X} \in \mathscr{C}_T$ as parameter: we call a process $g$ {\em causal in} $\BX$  (short: $\BX$-causal) if
\begin{equation} \label{equ:causal} 
g(t,\omega; \BX) = g(t,\omega;\BX_{\cdot \wedge t}). 
\end{equation}

 \subsection{Function spaces}  

For some normed vector spaces \(K,\bar{K}\), we denote by \(\C_b(K,\bar K)\) the space continuous bounded functions from \(K\) to \(\bar K\) equipped with the norm
\begin{align*}
   f\mapsto |f|_\infty=\sup_{y\in K}|f(y)|.
\end{align*}
For each real number  \(\gamma>0\), we write $\C_b^\gamma (K;\bar K)$ for the classical Lipschitz space of functions from \(K\) to \(\bar K\) with regularity $\gamma$. 
More precisely, for $\gamma=N+\beta$ where $N$ is a non-negative integer and $0<\beta\le 1$, $\C^\gamma_b(K;\bar K)$ consists of bounded functions $f\colon K\to \bar K$ such that $f$ has Fr\'echet derivatives up to order $N$, $D^jf$, $j=1,\ldots,N$ are bounded functions  and $D^Nf$ is globally H\"older continuous with exponent $\beta$.
	For each $f$ in $\C_b^\gamma$, we denote
	\[
			[f]_\gamma=\sum\nolimits_{k=1}^N|D^kf|_\infty
         +\sup_{x,y\in K}\frac{|D^N f(x)-D^Nf(y)|}{|x-y|^\beta}
			\tand |f|_\gamma=|f|_\infty+[f]_\gamma.
	\]

\section{Snapshots of RDEs and RSDEs}   
\label{sec:RSDE}
We briefly review the main results for rough (stochastic) differential equations that are closely relevant to our study. Our goal is to present simpler statements with more straightforward conditions while still achieving sufficiently general conclusions.

\subsection{Rough paths and its integration} \label{sec:RDE}

The purely deterministic theory of rough paths \cite{MR1654527} 
gives well-posedness to rough differential equations (RDEs) of the form
\begin{equation}\label{eq:RDE}
		dY_t=
		b_t (Y_t)dt + f_t(Y_t)d\BX_t 
		,\quad t\in[0,T].
\end{equation}
Here $\BX = (X,\XX) \in \mathscr{C}^\alpha_T$ is an $\alpha$-H\"older rough path, $\alpha \in (1/3,1/2]$.

The solution $Y=Y^\BX$ is an example of a {\em controlled rough path} \cite{MR2091358}, or simply {\em controlled} path (w.r.t. $X$), in the sense that it looks like $X$ on small scales: $Y_t \approx Y_s + Y_s' (X_t - X_s)$, with $Y'_s := f_s (Y_s)$. 
Crucially, the definition of $\int f (Y) d \BX$, and then integral meaning to \eqref{eq:RDE}, requires $f(Y)$ itself to be controlled. 
Many works on this subject, including \cite{FH20}, consider autonomous situation where $f_t(\cdot) \equiv f(\cdot)$, but there is no difficulty in assuming 
$f_t (\cdot) \approx f_s(\cdot) + f'_s (\cdot) (X_t -X_s)$ to accommodate (controlled) rough time dependence, in which case we write 
$$
dY_t=
		b_t (Y_t)dt + (f_t,f_t')(Y_t)d\BX_t .
$$
With $Y'_s = f_s (Y_s)$ and $Y''_s = ((Df_s) f_s +f'_s) (Y_s)$, one expects a Davie-type expansion
$$
         Y_t = Y_s + Y_s' (X_t - X_s) + Y''_s \XX_{s,t} + o(t-s).
$$ 
In quantitative form, cf. \cite{MR2387018,FH20}, this characterizes RDE solutions.
\begin{rem} \label{rem:nodimRDE}
We tacitly assume that $X$ and $Y$ take values in  $\mathbb{R}^{d_X}$ and $\mathbb{R}^{d_Y}$, respectively. For instance, a full specification of $\BX$ (``over $X$'') would then read $\Ca ( [0,T],\mathbb{R}^{d_X})$, but we prefer the short notation $\Ca_T$. Similarly, we avoid full specifications like
$b(t, . ): \mathbb{R}^{d_Y} \to \mathbb{R}^{d_Y}$ for (time-dependent) vector fields, or $f(t,. ): \mathbb{R}^{d_Y} \to \mathrm{Lin}(\mathbb{R}^{d_X}, \mathbb{R}^{d_Y})$. Indeed, explicit information about the dimensions only enter trivially in our arguments, e.g. in dealing with summations implicit in expressions like $Y_s' (X_t - X_s)$ or $Y''_s \XX_{s,t}$).
\end{rem}

 Important examples of rough paths come from the typical realisation of a multidimensional Brownian motion enhanced with iterated (It\^o) integrals,
\[
	\BX = (X, \X) = \left(B(\omega), (\int \delta B \otimes d B) (\omega) \right)=: \mathbf{B}^{\text{It\^o}} (\omega).
\]
As is well-known, e.g. \cite[Ch.9]{FH20}, under natural conditions, $\bar{Y} (\omega) := Y^\BX|_{\BX = \mathbf{B}^{\text{It\^o}} (\omega)}$ yields a (beneficial\footnote{See e.g. \cite{FV10, FH20} for applications to large deviations, support theorems, H\"ormander theory etc, all of which are evidence of the benefits of the RDE approach to SDEs.}
) version of the It\^o solution to $dY_t = b_t (Y_t) dt + \sigma_t (Y_t) dB_t$,
 the Stratonovich case is similar.

\subsection{Rough stochastic differential equations (RSDEs)}

Let $(\Omega,\cff,(\cff_t)_{t \in [0,T]},\mathbb{P})$ be a complete filtered probability space that supports a $\mathbb{R}^{d_B}$-dimensional Brownian motion $B$.
It was an open problem until \cite{FHL21} to provide a unified approach to SDEs and RDEs, such as to give intrinsic meaning and well-posedness to {\em rough stochastic differential equations} (RSDEs), aiming for an adapted solution process $Y=Y^{\BX} (\omega)$ to
\begin{equation}\label{eq:dRDE}
       d Y_t = b_t (Y_t)dt + \sigma_t ( Y_t) dB_t + (f_t,f_t')(Y_t) d \BX_t .
\end{equation}
The coefficients \(b,\sigma,f,f'\) are allowed to be progressively measurable.
As before, $\BX \in \mathscr{C}^\alpha_T$ with  $\alpha \in (1/3,1/2]$ and again (cf.  \cref{rem:nodimRDE}) we need not be explicit about the dimensions. The aforementioned paper provides both a Davie expansion, which quantifies 
$$
\delta Y_{s, u} \approx \int_s^u b_t (Y_t) dt + \int_s^u \sigma_t (Y_t) dB_t
+ f_s (Y_s) \delta X_{s, u} + ((D f_s) f_s + f'_s) (Y_s) \mathbb{X}_{s, u},
$$ and a rough stochastic integration theory which gives intrinsic meaning and estimates for rough stochastic integrals, 
\begin{equation}  \label{eqn:rsi}
      \int_s^u (Z, Z') d \BX \approx Z_s \delta X_{s, u} + Z'_s\mathbb{X}_{s, u}
\end{equation} 
where $(Z,Z')$ is a so-called stochastic controlled rough path (see \cref{def:stochasticcontrolledroughpaths} below).
It is given by $Z=f(Y),Z' = ((D f) f + f')(Y)$ in the above example and provides integral meaning to \eqref{eq:dRDE}.
It is an important feature of that only $f$ needs to be Taylor-expanded, whereas $b,\sigma$ can have progressive $(t,\omega)$ dependence and only need to have Lipschitz spatial regularity, as in classical It\^o theory, to have well-posedness.
\begin{rem}
 
The progressive generality (in $b,\sigma)$ is crucial if one considers stochastic control problems. A  commonly used trick is to regard \eqref{eq:dRDE} as RDE driven by a random rough path,``$(B,\BX)$'' $\in \mathscr{C}([0,T];\mathbb{R}^{d_B + d_X})$ obtained by a joint lift of $\mathbb{R}^{d_B}$-dimensional Brownian motion $B$  and $\BX \in \mathscr{C}([0,T];\mathbb{R}^{d_X})$, essentially by supplying the missing integrals $\int \delta B d B, \int \delta X dB$ by It\^o integration, $\int \delta B dX$ via integration by parts.
However, this method fails to treat such general setup: a progressive path-dependent coefficient field $\sigma (t,\omega)$, comes with no H\"older or $p$-variation regularity in $t$ whatsoever, leave alone controlled $t$-dependence, which prohibits the use of this joint lifting method.

\end{rem}

We now review some key material of \cite{FHL21}, to the extent needed later on.
Let \(p \in [2,\infty)\) and \(q \in [p,\infty]\), and let \(\mathfrak{G}\subseteq\cff\) be a sub-\(\sigma\)-field. 
Given a random variable \(\xi\), we define (if it exists) its conditional \(L^p\)-norm with respect to \(\mathfrak{G}\) as the (unique) \(\mathfrak{G}\)-measurable \(\mathbb{R}\)-valued random variable given by 
\[
    \|\xi|\mathfrak{G}\|_{p} := \mathbb{E}\left( |\xi|^p|\mathfrak{G}\right)^{\frac{1}{p}}.
\]
In particular, the mixed \(L_{p,q}\)-norm \(\|\|\xi|\mathfrak{G}\|_p\|_q\) denotes the \(L^q(\Omega)\)-norm of \(\|\xi|\mathfrak{G}\|_p\).
Let \(\kappa,\kappa' \in [0,1]\) and let \(X :[0,T]\to \R^{d_X}  \) be an \(\alpha\)-H\"older continuous path.

\begin{defi}[{\cite[Definition 3.1]{FHL21}}] \label{def:stochasticcontrolledroughpaths}  
A pair \((Z,Z')\) is called \textit{stochastic controlled rough path} of $(p,q)$-integrability and $(\kappa,\kappa')$-H\"older regularity if 
\begin{enumerate} \renewcommand{\labelenumi}{\alph{enumi})}
\item \(Z=(Z_t)_{t \in [0,T]}\) is progressively measurable such that
\[
    \|\delta Z\|_{\kappa;p,q} := \sup_{0 \leq s < t \leq T} \frac{\|\|\delta Z_{s,t}|\cff_s\|_p\|_q}{|t-s|^\kappa} < +\infty;
\]
\item \(Z'=(Z'_t)_{t \in [0,T]}\) is progressively measurable such that
\[
    \sup_{t \in [0,T]} \|Z'_t\|_q < +\infty \quad \text{and} \quad \|\delta Z'\|_{\kappa';p,q} := \sup_{0 \leq s < t \leq T} \frac{\|\|\delta Z'_{s,t}|\cff_s\|_p\|_q}{|t-s|^{\kappa'}} < +\infty;
\]
\item writing \(R^Z_{s,t} = \delta Z_{s,t} - Z'_s \delta X_{s,t}\) for \((s,t)\in\Delta_T\), we have
\[
    \|\mathbb{E}_\cdot R^Z\|_{\kappa+\kappa';q} := \sup_{0 \leq s < t \leq T} \frac{\|\mathbb{E}_s(R^Z_{s,t})\|_q}{|t-s|^{\kappa+\kappa'}} < + \infty.
\]
\end{enumerate}
We write \((Z,Z') \in \mathbf{D}_X^{\kappa,\kappa'} L_{p,q}\); and also \(\mathbf{D}_X^{2\kappa} L_{p,q}\), \(\mathbf{D}_X^{\kappa,\kappa'} L_p\) if \(\kappa = \kappa'\) and \(p=q\), respectively.
\end{defi}


A seminorm is then defined by 
\[
    \|(Z,Z')\|_{\mathbf{D}_X^{\kappa,\kappa'}L_{p,q}} :=  	
    \|\delta Z\|_{\kappa;p,q}     
     + \sup_{t \in [0,T]}\|Z'_t\|_q 
   + \|\delta Z'\|_{\kappa';p,q} + \|\mathbb{E}_\cdot R^Z\|_{\kappa+\kappa';q}.
\]
Given additionally \((\bar{Z},\bar{Z}') \in \mathbf{D}_{\bar{X}}^{\kappa,\kappa'} L_{p,q}\) for some \(\bar{X} \in C^\alpha([0,T];\mathbb{R}^{d_X})\), we define\footnote{Following \cite{FHL21} we only use the $L_p$-scale here, even if $\bar{Z},\bar{Z}'$ enjoy partially better integrability.}
\begin{align*}
    \|Z,Z';\bar{Z},\bar{Z}'\|_{X,\bar{X};\kappa,\kappa';p} &:= 
    \|\delta (Z-\bar{Z})\|_{\kappa;p,p}  
    +\| Z'-\bar{Z}'\|_{\kappa';p,p} + \|\mathbb{E}_\cdot R^Z - \mathbb{E}_\cdot \bar{R}^{\bar{Z}}\|_{\kappa,\kappa';p},
\end{align*} 
where \(\bar{R}^{\bar{Z}}_{s,t} = \delta \bar{Z}_{s,t} - \bar{Z}'_s \delta \bar{X}_{s,t}\).
If \(\bar{X} = X\), write \(\|\ \cdot \ ; \ \cdot \ \|_{X;\kappa,\kappa';p}\) instead of \(\|\ \cdot \ ; \ \cdot \|_{X,X;\kappa,\kappa';p}\).
\medskip
When \(X\) has a rough path lift \(\BX \in \CC^\alpha\), the integration of \((Z,Z')\) against \(\BX\) is a rough  stochastic integral defined rigorously in the following result, which is an excerpt from {\cite[Theorem 3.5]{FHL21}}.
\begin{proposition}\label{prop.rsint} 
   Suppose that \((Z,Z')\) belongs to \(\mathbf{D}_X^{\kappa,\kappa'} L_{p,q}\) with \(\alpha+\kappa>1/2\), \(\alpha+\min(\alpha,\kappa)+\kappa'>1\). Then the Riemann sums
   \begin{align*}
      \sum_{[u,v]\in\cpp, u\le t}Z_u \delta X_{u,v\wedge t}+Z'_u\XX_{u,v\wedge t}
   \end{align*}
   converge uniformly in time \(t\in[0,T]\) in \(L_p\), as \(|\cpp|\) goes to 0, to a continuous adapted process, denoted by \(\int_0^\cdot (Z,Z')d\BX\). 
\end{proposition}
We now define integrable solutions to  \eqref{eq:dRDE}.
\begin{defi}[{\cite[Definition 4.2]{FHL21}}] \label{def:solutionRSDEs_compendium}
 An $L_{p,q}$-integrable solution of \eqref{eq:dRDE} on $[0,T]$ is a continuous adapted process $Y$ such that the following conditions are satisfied:
     \begin{enumerate}\renewcommand{\labelenumi}{\alph{enumi})}
        \item $\int_0^T |b_r(Y_r)|dr$ and $\int_0^T |(\sigma \sigma^\top)_r(Y_r)|dr$ are finite $\mathbb{P}$-a.s.;
        \item $(f(Y),D_y f(Y)(f(Y))+f'(Y))$ belongs to $\mathbf{D}_X^{\bar\alpha,\bar\alpha'}L_{p,q}$,
        for some $\bar\alpha,\bar\alpha'\in(0,1]$ such that $\alpha+\min\{\alpha,\bar\alpha\}>\frac12$ and $\alpha+\min\{\alpha,\bar\alpha\}+\bar\alpha'>1$; 
        \item $Y$ satisfies the following stochastic Davie-type expansion for every $(s,t)\in \Delta_T$: \begin{equation*} \begin{aligned}
            \delta Y_{s,t} &=  \int_s^t b_r(Y_r)dr + \int_s^t \sigma_r(Y_r)dB_r  \\
            & \quad + f_s(Y_s) \delta X_{s,t} + (D_y f_s(Y_s)(f_s(Y_s))+f'_s(Y_s))\mathbb{X}_{s,t} + Y^\natural_{s,t},
        \end{aligned} 
        \end{equation*} where 
        \begin{equation*}
           \|\|Y^\natural_{s,t}|\cff_s\|_p\|_q=o(|t-s|^\frac{1}{2}) \qquad \text{and} \qquad \|\mathbb{E}_s(Y^\natural_{s,t})\|_q=o(|t-s|).
        \end{equation*}
     \end{enumerate}
     When the starting position $Y_0=\xi$ is specified, we say that $Y$ is a solution starting from $\xi$. 
       Furthermore, by \cite[Proposition 4.3]{FHL21}, c) can be replaced by \begin{itemize}
           \item[c')] $\mathbb{P}$-almost surely and for any $t \in [0,T]$, \begin{equation*}
               Y_t = \xi + \int_0^t b_r(Y_r) dr + \int_0^t \sigma_r(Y_r) dB_r + \int_0^t (f_r(Y_r),D_y f_r(Y_r)(f_r(Y_r)) + f'_r(Y_r)) d\mathbf{X}_r. 
           \end{equation*} 
       \end{itemize}
 \end{defi}

To state the well-posedness result for \eqref{eq:dRDE}, we recall some terminologies related to the regularity of vector fields.
\begin{definition}[ {\cite[Definition 4.1]{FHL21}}] \label{def:stochasticboundedandLipschitzvectorfield}
   Let $V,\bar V$ be some finite dimensional Euclidean spaces and fix a Borel set \( S\subset V \). Let \( (t,\omega)\mapsto g_t(\omega,\cdot) \) be a progressively measurable stochastic process from \( \Omega\times [0,T] \to \Cb(S;\bar V)\).
We say that:
   \begin{enumerate}[label=(\alph*)]
      \item $g$ is \textit{random bounded continuous} if it is uniformly bounded, namely, there exists a deterministic constant $\|g\|_\infty$ such that
         \[
            \sup_{t\in[0,T]}\esssup_{\omega\in \Omega}\sup_{x\in S}|g_t(\omega,x)|\le\|g\|_\infty.
         \]
      \item $g$ is \textit{random bounded Lipschitz} if it is random bounded continuous, progressively measurable from \( \Omega\times[0,T]\to \Cb^1 (S;\bar V) \) and uniformly bounded in the sense that
         \[
            \sup_{t\in[0,T]}\esssup_{\omega\in \Omega}\sup_{x,\bar x\in S}\frac{|g_t(\omega,x)-g_t(\omega,\bar x)|}{|x-\bar x|}\le \|g\|_\lip
         \]
         for some constant \( \|g\|_{\lip} \).
   \end{enumerate}
\end{definition}


\begin{definition}[Stochastic controlled vector fields, {\cite[Definition 3.8]{FHL21}}]\label{def.scvec}
   Let $\beta, \beta' \in (0,1]$ and $\gamma >1$ be some fixed parameters.
   We call $(f,f')$ {\em stochastic controlled vector field on \( \R^d \)} and write \( (f,f')\in \D^{\beta,\beta'}_XL_{p,q}\C^\gamma_{b}\)
   if the following conditions are satisfied.
   \begin{enumerate}[label=(\alph*)]
      \item\label{scvec.f}
      The pair
      \[
      (f,f')\colon\Omega \times  [0,T] \to \C^\gamma_{b}(\R^{d}; \R^m) \times \C^{\gamma-1}_{b} (\R^d ; \mathrm{Lin}(\R^{d_X},\R^m))
      \]
      is  progressively measurable in the strong sense and uniformly $n$-integrable in the sense that
      \begin{align}\label{def.normff}
         \|(f,f')\|_{\gamma;q}:=\sup_{s\in [0,T]}\||f_s|_{\gamma}\|_q +  \sup_{s\in [0,T]}\||f'_s|_{\gamma-1}\|_q 
      \end{align}
      is finite.

      \item\label{scvec.brakets}
      Letting
      \begin{align}\label{def.bk}
         \bk{Z}_{\zeta;p,q}:=\sup_{(s,t)\in \Delta_T:s\neq t}\frac{\Big\|\big\|\sup_{y\in \Rd}|Z_{s,t}(y)|\,\big|\,\cff_s \big\|_p\Big\|_q}{(t-s)^\zeta}\,,
      \end{align}
      the quantities
      $\bk{\delta f}_{\beta;p,q}$, $\bk{\delta f'}_{\beta';p,q}$, \( \bk{\delta Df}_{\beta';p,q} \) are finite.%

      \item\label{scvec.remainder} The map $(s,t)\mapsto \E_s R^f_{s,t}=\E_sf_t-f_s-f'_s \delta X_{s,t}$ satisfies 
      \[
      \bk{\E_\bigcdot R^f}_{\beta+\beta';q}
      =\sup_{(s,t)\in \Delta_T:s\neq t}\frac{\|\sup_{y\in \Rd}|\E_sR^{f}_{s,t}(y)|\|_q}{(t-s)^{\beta+\beta'}}
      <\infty\,.
      \]
   \end{enumerate}
\end{definition}
For stochastic controlled vector fields as above, we introduce the quantity
\begin{equation}
   \label{def.norms_scvf}
   \bk{(f,f')}_{X;\beta,\beta';p,q}
      :=\bk{\delta f}_{\beta;p,q}+\bk{\delta Df}_{\beta';p,q}+\bk{\delta f'}_{\beta';p,q}+\bk{\E_\bigcdot R^f}_{\beta+\beta';p,q},
\end{equation}
which is abbreviated as $\bk{(f,f')}_{X;\beta,\beta';p}$ if $p=q$, as $\bk{(f,f')}_{X;2\beta;p,q}$ if \( \beta=\beta' \) and as $\bk{(f,f')}_{X;2\beta;p}$ when both conditions are met.
Furthermore, if $(\bar f,\bar f')\in \D^{\beta,\beta'}_{\bar X}L_{p,\infty}\C^\gamma_b$ for another such \(\bar X\in C^\alpha(V) \), we define 
\begin{equation}
\label{def.bk_metric}
\begin{aligned}
   \bk{f,f';\bar f,\bar f'}_{X,\bar X;\beta, \beta';p}&=\bk{\delta f-\delta \bar f}_{\beta;p,p}+\bk{\delta f'-\delta \bar f'}_{\beta';p,p}+\bk{\delta Df-\delta D\bar f}_{\beta';p,p}
   \\&\quad+\bk{\E_\bigcdot R^f-\E_\bigcdot \bar R^{\bar f}}_{\beta+\beta';p}.
\end{aligned}
\end{equation}

\begin{definition}[Controlled vector fields]
In the case of deterministic vector fields, stochastic controlled vector fields reduce to controlled vector fields and we write \(\D^{\beta,\beta'}_X\C^\gamma_{b}\) instead of \(\D^{\beta,\beta'}_XL_{p,q}\C^\gamma_{b}\). In such cases, we will also drop the indices \(p,q\) in the notations in \eqref{def.normff}-\eqref{def.bk_metric}.
\end{definition} 

The following results on well-posedness of \eqref{eq:dRDE} are consequences of more general ones from \cite{FHL21}. 

\begin{thm}[see  {\cite[Theorem 4.6 and Proposition 4.5]{FHL21}}] \label{thm:wellposed}
Let $\BX \in \mathscr{C}^\alpha, \alpha \in (1/3,1/2]$ and $p \in [2,\infty)$.
    Assume\footnote{For simplicity, take \(\beta=\alpha\) at first reading.}  that $\beta \in [0,\alpha]$ and $\gamma \in \left(\frac{1}{\alpha},3 \right]$ such that $\alpha + (\gamma-1)\beta > 1 $. Let $b$ and $\sigma$ be random bounded Lipschitz vector fields.
    Let $(f, f')$ be a stochastic controlled vector field in \(\mathbf{D}_X^{2\beta} L_{p,\infty} \C_b^\gamma\) such that \((D_yf,D_yf')\) belongs to \(\mathbf{D}_X^{2\beta} L_{p,\infty} \C_b^{\gamma-1}\).
    Then, for any $\xi \in L^p(\Omega,\cff_0;\mathbb{R}^d)$, there exists a unique \(L_{p,\infty}\)- integrable solution $Y$ to \eqref{eq:dRDE}.
    Moreover, 
    \begin{align}\label{est.apri.m}
      \|(Y,Y')\|_{\mathbf{D}_X^{\alpha,\beta}L_{p,\infty}}
      &\lesssim_{T,p,\alpha,\beta,\gamma}
      (1+M)K^{\,2+2/((\gamma-2)\beta)},
    \end{align}
    where \(M >0 \) is any constant such that \(\| (f, f') \|_{\gamma-1;\infty} + \llbracket(f, f') \rrbracket_{X;2\beta; p,\infty} \leq M\).
    Here $K=1+\|b\|_\infty+\|\sigma\|_\infty+M\nn{\BX}_\alpha+M^2\nn{\BX}_\alpha^2$.
\end{thm}

\begin{thm}[see {\cite[Theorem 4.9]{FHL21}}] \label{thm:stabilityforRSDEs_compendium}
     
    In the setting of \cref{thm:wellposed}, let \(\bar{\mathbf{X}} = (\bar{X}, \bar{\mathbb{X}}) \in \mathscr{C}^\alpha([0,T]; \mathbb{R}^{d_X})\), let \(\bar{b}\) and \(\bar{\sigma}\) be two random bounded continuous vector fields 
    and let \((\bar{f}, \bar{f'}) \in \mathbf{D}_{\bar{X}}^{2\beta} L_{p,\infty} \C_b^\gamma\).
    Denote by \(\bar{Y} = (\bar{Y}_t)_{t \in [0,T]}\) the solution of the rough SDE starting from \(\bar{\xi} \in L^p(\Omega,\cff_0;\mathbb{R}^d)\) and driven by \(\bar{\mathbf{X}}\), with coefficients \(\bar{b}, \bar{\sigma}\) and \((\bar{f}, \bar{f'})\). Then 
    \[
    \begin{aligned}
      \|\sup_{t\in[0,T]}| Y_{t}- \bar Y_{t}|\|_p&+ \|Y, f(Y); \bar{Y}, \bar{f}(\bar{Y})\|_{X, \bar{X}; \alpha, \beta; p} \\
      &\lesssim_{M, \alpha, \beta, p, T} \|\xi - \bar{\xi}\|_p + \rho_\alpha(\mathbf{X}, \bar{\mathbf{X}}) 
      \\&\quad+ \sup_{t \in [0,T]} \| |b_t(\cdot) - \bar{b}_t(\cdot)|_\infty \|_p 
      +\sup_{t \in [0,T]} \| |\sigma_t(\cdot) - \bar{\sigma}_t(\cdot)|_\infty \|_p
      \\&\quad + \|(f - \bar{f}, f' - \bar{f'}) \|_{\gamma-1; p} 
       + \llbracket f, f'; \bar{f}, \bar{f}' \rrbracket_{X, \bar{X}; 2\beta; p},
    \end{aligned}
    \]
    where \(M > 0\) is any constant such that
\begin{multline*}
   \nn{\BX}_\alpha + \nn{\bar \BX}_\alpha + \|b\|_{\lip} + \|\sigma\|_{\lip} 
   \\+ \| (f, f') \|_{\gamma; \infty} + \llbracket (f, f') \rrbracket_{X; 2\beta; p; \infty}+\bk{(Df,Df')}_{X;2\beta;p,\infty} \leq M.
\end{multline*}
\end{thm}

\section{Measurable selection for RSDEs} 
\label{sec:measurable_prelim}

 When solving RSDEs, the rough path is fixed and plays the role of a deterministic parameter. 
In the randomisation step, one needs to substitute the rough path by a random process which requires certain measurability properties of the solutions to RSDEs with respect to the rough path input.
We provide answers to these measurability issues in this section, building upon the  
classical results from Stricker--Yor 
\cite{SY78}, which we revise for the reader's convenience in the appendix.

Let \((U,\uuu)\) be a measurable space and \((M,\mmm)\) be a Polish measurable space. In later sections, the abstract measurable parameter space $(U, \uuu)$ 
is taken as the rough path space $\MC^{\alpha}$, equipped with its Borel sets, $(M,\mmm)$ is often just $(\R^d , \bbb^d)$. 
\begin{defi} \label{def:optional} Let $(\Omega, \cff, (\cff_t)_{t \geq 0}, \P)$ be a filtered probability space satisfying the usual conditions. 
A measurable map     
$$
Y:( [0,T] \times \Omega \times U , \bbb_T  \otimes \cff \otimes  \uuu ) \to (M,\mmm) , \quad (t,\omega,u) \mapsto Y^u(t,\omega)
$$
is called {\em $\uuu$-optional} 
 if it is $\ooo \otimes \uuu  /  \mmm$-measurable where $\ooo$ denotes the optional $\sigma$-field on $[0,T] \times \Omega$. We say that $Y$ has a {\em $\uuu$-optional version} if there exists  $ ^{\circ}Y$
 which is $\uuu$-optional and {\em indistinguishable} from  $Y$, in the sense that, for all $u \in U$, the processes $ ^{\circ} Y^u$ and $Y^u$ are indistinguishable. 
\end{defi}

Thanks to Theorem \ref{thm:wellposed}, under the conditions stated there, there exists a unique \(L_{p,\infty}\)-integrable 
 solution on $[0,T]$ to
\footnote{As earlier, $(f_r, f_r') (Y_r ; \BX) = (f_r (Y_r ; \BX), ((D_y f_r)
f_r) (Y_r ; \BX) + f_r' (Y_r ; \BX ) )$.}

\be\label{rand-rsde}
        Y_t = \xi  + \int_0^t b_r(Y_r; \BX) \, dr + \int_0^t \sigma_r(Y_r; \BX)\, dB_r + \int_0^t \Big(f_r, 
         f'_r \Big)(Y_r; \BX) d\mathbf{X}_r, 
\ee
 for any fixed rough path $\BX \in \Ca_T = \mathscr{C}^{\alpha}([0,T])$. To show that the solution is measurable with respect to \(\BX\), we allow  $\xi$ and $g \in \{ b, \sigma, f,f' \}$ to 
  depend on $\BX$ in a measurable way. Specifically, we assume that $g=g(t,\omega,y,\BX)$ to be $\bbb^{d_Y} \otimes \BCa_T$-optional
  in the above sense with $U = \R^{d_Y} \times \mathscr{C}_T^{\alpha}$, equipped with the product $\sigma$-field $ \bbb^{d_Y} \otimes \BCa_T $. (All assumptions on $g=g(\cdot,\cdot,\cdot,\BX)$ for fixed $\BX$ imposed by Theorem \ref{thm:wellposed} remain.) These conditions are stated precisely in \cref{assum-jt-measu} below. Note that, at this stage, we do not assume that $g$ is causal in $\BX$.
 

\begin{exam} 
Consider a control $\eta = \eta (t,\omega,\BX)$ which is $\BCa_T$-optional, with values in some Polish measurable space $(A,\aaa)$,
and measurable  fields $g \in \{ b, \sigma, f, f' \}$, with $  g(t,y,a)$ defined on $[0,T] \times \R^{d_Y} \times A$.
Then $g(t,\ome,y;\BX):= g(t,y,\eta(t,\ome,\BX))$ is $\bbb^{d_Y} \otimes \BCa_T$-optional.

\end{exam} 



\begin{assumption}\label{assum-jt-measu}  Assume that $\xi:\Omega   \times \Ca_T \rightarrow \R^{d_Y}$ is $\MF_0 \otimes \BCa_T/\bbb^{d_Y}$-measurable, and for any $\BX \in \Ca_T,$ assume $\xi(\cdot,\BX) \in L_p$. Assume that $g\in \{b,\sigma,f,f'\}$ defined on $[0,T] \times \Omega \times \R^{d_Y} \times \Ca_T$ is $\bbb^{d_Y} \otimes \BCa_T$-optional. Moreover, we assume that the following conditions hold.

\begin{itemize}

	\item[$(1)$.] $(b,\sigma)$ is bounded and uniformly Lipschitz in $y$ in the sense that for any $y_1, y_2 \in \R^{d_Y},$ $g\in \{b,\sigma\},$
	\[
	\sup_{(t,\BX)\in [0,T]\times \Ca_T}\esssup_{\omega\in \Omega}
	\sup_{y_1\ne y_2}
	\frac{|g(t,\omega,y_1;\BX)-g(t,\omega,y_2;\BX)|}{|y_1-y_2|}
	\le \|g\|_\lip
	\]
   for some positive constant \(\|g\|_\lip\).

\item[$(2)$.] For any $\BX \in \Ca_T$, $(f,f') (\cdot \ ;\BX) $ belongs to $ \BD_X^{2\beta }L_{p,\infty} \C^{\gamma}_b$ and $(D_yf,D_yf')(\cdot\ ;\BX)$ belongs to $\BD_X^{2\beta}L_{p,\infty}\C_b^{\gamma-1}$ for some $p \ge 2,$ $\beta \in (0,\alpha]$, ${3 \ge} \gamma > \frac{1}{\alpha}$ such that $\alpha+\beta>\frac12$ and $\alpha + (\gamma-1)\beta>1.$
\end{itemize}	
\end{assumption}

\begin{exam}
We give examples of $(f,f')$ which satisfy  \cref{assum-jt-measu}, (2). 
\begin{enumerate}[(i)]
   \item If there is a finite constant \( C\) so that additionally \(f\)  satisfies 
   \begin{align*}
      \esssup_{\omega} \sup _{y}|f(t,\omega,y)-f(s,\omega,y)| \le C(t-s)^{2 \beta}
   \end{align*}
   for every  \((s,t)\in \Delta_T\), then \((f,f'):=(f,0)\)  belongs to $ \BD_X^{2\beta }L_{p,\infty} \C^{\gamma}_b$.
	\item Let $h:  \R^{d_Y} \times \R^{d_X} \rightarrow \R^d$ be in $\C^3_b$. Then $$(t,\ome,y,\BX) \mapsto (f,f')(t,\ome,y;\BX) := ( h(y ;X_t), D_x h(y; X_t) ) $$ belongs to $ \BD^{2\alpha}_X \C^3_b$.
	\item Let $\ell : \R^{d_B} \rightarrow \R$ be in $\C^{3}_b.$ Then $(t,\ome,y,\BX) \mapsto (f,f')(t,\ome,y;\BX):= (\ell(B_t(\ome)), 0) $ belongs to $ \BD_X^{1/2,1/2} L_{2,\infty} \C_b^{\gamma}$ (equivalently in our notation, \(\BD_X^{1} L_{2,\infty} \C_b^{\gamma}\))  for any \(\gamma>0\). 
\end{enumerate}

\end{exam}

The following result is the main result of the current section, which shows that the solution to \eqref{rand-rsde} has a \(\BCa_T\)-optional version. 
\begin{thm}[Measurable selection] \label{thm-rsde-optional} Suppose that  \cref{assum-jt-measu} holds. For each \(\BX\in \CC_T\), let $Y=Y  (t,\omega; \BX)$ be the solution to \eqref{rand-rsde} on $[0,T]$ which is supplied by \cref{thm:wellposed}. Then $Y$ has a
$\BCa_T$-optional version (denoted by $\tY$) which is $\P$-a.s. continuous for any fixed $\BX \in \MC_T$. 
In particular, $\tY_T$ is $ \cff_T \otimes \BCa_T / \bbb^{d_Y}$-measurable. 
 \end{thm}
%

We prepare the proof with the following measurable selection result, closely related to \cite{SY78} which we include for the reader's convenience as  \cref{SY_Prop1} in the appendix (see also  \cite[Theorem 62]{MR2273672}). 

 \begin{lem}\label{optional-limit}
 	Suppose that $Z^n:[0,T] \times \Omega \times U \rightarrow \R^d$ has a $ \uuu$-optional version in the sense of Definition \ref{def:optional}, for each $n \in \N$. Assume for each $u \in U, $ $\{Z^n(\cdot, u)\}_n$ are c\`adl\`ag, and converge to $Z(\cdot,u)$ uniformly in time in $\P$-probability, where $Z:[0,T] \times \Omega \times U \rightarrow \R^d$.  Then $Z$ has a $ \uuu$-optional version, which is $\P$-a.s. c\`adl\`ag for each $u \in U.$ 
\end{lem}
 
\begin{proof}  Without loss of generality, we assume $Z^n$ is $ \uuu$-optional. 
 For each \(i,j\), define the process 
$$
\Delta_t^{i,j}(\ome, u):=|Z^i-Z^j|(t,\ome,u)
$$
which is $\uuu$-optional.  Put $n^u_0:=1, $ and for any $k \ge 1,$ define 
$$
n_k^u:= \inf \Big\{ n> (k \vee n_{k-1}^u): \sup_{i,j \ge n} \P(\sup_{t\in[0,T]}\Delta^{i,j}_t(u) > 2^{-k} ) \le 2^{-k} \Big\}. 
$$
Since for any $u$, $\Delta^{i,j}(u)$ is c\`adl\`ag and $\Delta^{i,j}$ is $\uuu$-optional, the mapping $$(u,n) \mapsto \sup_{i,j \ge n} \P(\sup_{t\in[0,T]}\Delta^{i,j}_t(u) > 2^{-k} )$$ is $\uuu \otimes \bbb(\N)/ \bbb^{1}$-measurable. It follows that  the map \(u\mapsto n^u_k\) from \( U\) to \(\N\) is measurable,
and thus for each \(k\), the process \(\tZ^k\), defined by $\tZ^k_t(\ome,u):= Z^{n^u_k}_t(\ome,u)$ for any $(t,\ome,u)$, is $\uuu$-optional. Now let 
\be
\tZ_t(\ome,u):=\left\{
\begin{array}{ll}
	\lim_k \tZ^k_t(\ome,u)  , \ \ & \text{if the limit exists, for fixed $(t,\ome,u)$}, \\[2mm]
	0, \ \ &\text{otherwise.}
\end{array}\right.
\ee
Then $\tZ$ is $\uuu$-optional. 
For each $u \in U,$ since $\tZ^k(\cdot,u)$ converges uniformly in time to $Z(\cdot,u)$ in probability, we have $\P(\text{for any $t \in [0,T],$ } \tZ_t(\cdot,u)=Z_t(\cdot,u))=1$, and thus $\tZ$ is an $\uuu$-optional version of $Z.$ Moreover, since for each fixed $u,$ $Z^n(\cdot, u)$ is c\`adl\`ag and  converges to $Z$ uniformly in time, we see that $\tZ(\cdot, u)$ is $\P$-a.s. c\`adl\`ag.
 \end{proof}

\begin{coro}\label{optional-r.i.}
 
Suppose that for any $\BX=(X,\X) \in \Ca_T,$ $(Z,Z'){(\cdot\ ; \BX)} \in \BD^{\gamma,\gamma'}_X L_{p,q}$ with \(p\in[2,\infty)\), \(q\in[p,\infty]\), \(\gamma,\gamma'\in[0,1]\) such that \(\alpha+\gamma>1/2\) and \(\alpha+\min(\alpha,\gamma)+\gamma'>1\). Furthermore, assume that $(Z,Z')$ as a mapping on $[0,T] \times \Omega \times \Ca_T$ has a $\BCa_T$-optional version. Then the rough integral $\int_0^.(Z,Z')d\BX$ has a $\BCa_T$-optional version, which is $\P$-a.s. continuous for any fixed $\BX \in \Ca_T$.
 
\end{coro}

\begin{proof} 
 Without loss of generality, we assume $(Z,Z')$ is $\BCa_T$-optional. According to \cref{prop.rsint}, for any $\BX \in \Ca_T$, 
the process $\int_0^.(Z,Z')d\BX$ is  the uniform in time limit in probability of  
$$
I^{\op,(\BX)}_t:= \sum_{[u,v] \in \op, u \le t} Z_u(\cdot\ ; \BX) \delta X_{u,v \wedge t} + Z'_u(\cdot\ ; \BX) \X_{u,v \wedge t},
$$
as $|\op|$ goes to zero. Then the result follows by  \cref{optional-limit}.
\end{proof}

\begin{lem}\label{meas-coef}
Suppose that $g:[0,T] \times \Omega \times \R^d \times \Ca_T \rightarrow \R^m$ is $\bbb^{d} \otimes \BCa_T$-optional. 
Suppose $Y:[0,T] \times \Omega \times \Ca_T \rightarrow \R^d$ has a $\BCa_T$-optional version. Then $Z_t(\ome,\BX):= g(t,\ome,Y_t(\ome,\BX), \BX)$ has a $\BCa_T$-optional version.
\end{lem}

\begin{proof}
	Without loss of generality, assume $Y$ is $\BCa_T$-optional. Then let 
	$$(S_T,\MS):=(([0,T]\times \Omega) \times \Ca_T, \MO \otimes  \BCa_T ),$$
	 and thus $g:S_T \times \R^d \rightarrow \R^m$ is $\MS \otimes \bbb^d/\bbb^m$-measurable. Note that $Y$ is $\MS/\bbb^d$-measurable, and thus $Z$ is $\MS/\bbb^m$-measurable, which means $\BCa_T$-optional.
\end{proof}

\begin{proof}[Proof of  \cref{thm-rsde-optional}]
Since global-on-$[0,T]$ solutions are constructed by concatenation of local solutions, it suffices to check the stated measurability for local solutions
(with random initial data). 
Following Theorem 4.6 in \cite{FHL21}, these are constructed by Picard iteration in a space of $X$-stochastic controlled rough paths, 
started from the process $t\mapsto (\xi  +f_0(\xi)\delta X_{0,t},f_0(\xi)) =: (Y^{\BX;(0)}, Y^{\prime,\BX;(0)})=: \mathrm{Y}^{\BX;(0)}$.
Then we define inductively $\mathrm{Y}^{\BX;(n+1)} \equiv (Y^{\BX;(n+1)}, Y'^{,\BX;(n+1)}) := (Y^{\BX;(n+1)}, f(Y^{\BX;(n)}) ), $ where
\be\label{meas-picard}
\begin{split}
Y^{\BX;(n+1)}:=\ & \xi + \int_0^t b_r(Y^{\BX;(n)}_r ; \BX)dr + \int_0^t \sigma_r (Y^{\BX;(n)}_r ; \BX)dB_r\\
& +\int_0^t  \Big(f_r, (D_y f_r)  Y'^{,\BX;(n)}_r+ f'_r   \Big) (Y^{\BX;(n)}_r; \BX )d\BX_r.	
\end{split}
\ee
We can see inductively that
$(t,\omega, \BX) \mapsto \mathrm{Y}^{\BX;(n)} (t, \omega)= \mathrm{Y}^{(n)} (t, \omega; \BX)$ is $\BCa_T$-optional. Indeed, the case $n=0$ being obvious in view of $\cff_0$-measurability of $\xi$ and $f_0.$ To see that $\BCa_T$-optional measurability of $\mathrm{Y}^{(n)}$ implies the same for $\mathrm{Y}^{(n+1)}=({Y}^{(n+1)}, f(Y^{(n)}) )$, we have to deal with Lebesgue integration $\int b (...) dt$, It\^o integration $\int \sigma (...) dB$ and  stochastic rough integration $\int (...  ) d \BX$, where in the these three cases the integrand $ (...)$ is itself $\BCa_T$-optional in view of Lemma \ref{meas-coef}. The desired $\BCa_T$-optional measurability of the first two cases then follows from results on measurable selection for stochastic integration with parameters, more precisely Lemma \ref{SY78-Lem2}  and Proposition \ref{SY78-Prop5} in the appendix. It remains to understand that the rough stochastic integral as map
$$
      (t, \omega, \BX) \mapsto \int_0^t \left(f_r, (D_y f_r) Y'^{,\BX;(n)}_r + f'_r \right)(Y^{\BX; (n)}_r ;\BX) d\mathbf{X}_r
$$ 
is $\BCa_T$-optional. By \cref{assum-jt-measu} $(2)$ and 
Lemma \ref{meas-coef} again, we see that the integrand, abbreviated to $(Z,Z')$ is $\BCa_T$-optional. The $\BCa_T$-optional measurability of the rough integral then follows from Corollary \ref{optional-r.i.} above. This shows that all ${Y}^{(n)}$, $n=1,2,...$ are $\BCa_T$-optional. 

By Lemma \ref{optional-limit} we can then find $Y=Y(t,\omega;\BX)$ which is $\BCa_T$-optional and the limit in $n$, uniformly on $[0,T]$ and in probability, of ${Y}^{\BX;(n)}$ for all $\BX$, which completes the proof.
 \end{proof}

\begin{rem} If we want to work with c\`adl\`ag rough paths, we should work with ``$\BCa_T$-predictable'' rather than $\BCa_T$-optional integrands, already needed when including BV integrators, as
is the case in the semimartingale setting of Theorem 1 in {\cite{SY78}}.
\end{rem}

\section{Rough stochastic control} \label{DPP_for_RSDEs}

\subsection{Setup and assumptions} \label{sub.setupdpp}


Let $\Omega = \C([0,T];\R^{d_B})$, and $B$ the canonical process on the Wiener space $(\Omega,\MF, ( \MF_t )_{t \ge 0}, \P)$ with $\{\MF_t\}_{t\ge 0}$ as the augmented filtration generated by $B$.
Let $A$ be a  Polish space equipped with its Borel sigma algebra $\aaa$, and $\MA$ be the space of {\em admissible} controls, by which we mean {\em optional}\footnote{The use of optional controls is related to Lemma \ref{lem:op} below.} processes, i.e. 
\be\label{def:ad-cont}
\MA:=\{\ctrl: ([0,T]\times \Omega, \MO) \rightarrow (A,\aaa) \text{ is measurable} \}.
\ee

Suppose that $(b,\sigma, \ell): [0,T] \times \R^{d_Y} \times A \rightarrow \R^{d_Y} \times \ML(\R^{d_B}, \R^{d_Y}) \times \R,$ $g:\R^{d_Y} \rightarrow \R,$ and $f:[0,T] \times \R^{d_Y} \rightarrow \ML(\R^{d_X}, \R^{d_Y})$. For any $\BX \in \Ca_T,$ we consider the following controlled RSDEs: for any $0<s\le t \le T$, $\ctrl \in \MA$, and $y \in \R^{d_Y},$ 
\be\label{c-rsde}\left\{
\begin{aligned}
&dY_t(\omega)=b(t,Y_t(\omega),\ctrl_t (\omega))dt+\sigma(t,Y_t(\omega),\ctrl_t(\omega))dB_t(\omega)+(f,f')(t,Y_t(\omega)) d\BX_t,\\
 & Y_s = y,
\end{aligned}
\right.
\ee
and the cost function 
\be\label{cost-rp}
J (s,y, \ctrl;\BX):=\E\left[ g(Y_T^{(s,y,\ctrl, \BX)})+ \int_s^T \ell (r, Y_r^{(s,y,\ctrl, \BX)}, \ctrl_r)dr \right],
\ee
with $Y^{(s,y,\ctrl,\BX)}$, also written as $Y^{s,y} (\ctrl,\BX)$, the solution of \eqref{c-rsde}.
The rough stochastic control problem concerns about the minimisation of the cost functional, directly related to the rough value function
\be\label{vf-rsde11}
\mathcal V (s,y; \BX ) := \mathrm{inf}_{\ctrl  \in {\mathcal{A}}} J (s,y, \ctrl;\BX).
\ee
The coefficients $(b,\sigma,f,g,\ell)$, which are implicit in \eqref{vf-rsde11}, are the data for the above control problem. For the well-posedness of controlled RSDE \eqref{c-rsde} and the cost function, we make the following assumption.

\begin{remark}[Removal of running cost by state-augmentation] \label{rem:remove_running_cost}
We may assume $\ell = 0$ for otherwise it suffices to consider the enhanced
process $(Y,Z)$ with $Y$ as in \eqref{c-rsde} and additional dynamics 
$$
dZ_t (\omega) = \ell (t, Y_t(\omega),\ctrl_t(\omega)) dt;
$$
in terms of $(Y,Z)$, with drift $(b,\ell)$, the value function \eqref{cost-rp} can then be written as
$$
    \tilde{\mathcal{V}} (s,(y,0); \BX)  =
     \inf_{\ctrl } \E^{s, (y,0)} [\tilde g (Y_T,Z_T)].
$$ 
for some $\tilde g$ such that $\tilde g (Y_T,Z_T) =  g (Y_T) +  Z_T$. (The obvious choice $(y,z) \mapsto g(y)+z$ is not bounded. However assuming $g, \ell$ uniformly bounded, it is clear that $Z$ also stays uniformly bounded on compacts in time, so that is suffices to modify $(y,z) \mapsto g(y)+z$ for large $z$ such as have $\tilde g$ also bounded.)
\end{remark} 

\begin{assumption}\label{assum-dpp} \phantom{new line}

   \item[$(1)$] Suppose that $ (b,\sigma, \ell) :[0,T] \times \R^{d_Y} \times A \rightarrow \R^{d_Y} \times \mathrm{Lin}(\R^{d_B}, \R^{d_Y}) \times \R $ is bounded and measurable. Moreover, each $\psi\in \{b,\sigma, \ell\}$ is uniformly Lipschitz in the sense that for any $y,\bar{y} \in \R^{d_Y}$,
\be
 \sup_{u \in A}\sup_{t\in [0,T]}  |\psi(t,y,u)-\psi(t, \bar{y}, u)| \le \|\psi\|_{\lip} |y-\bar{y}|, \label{LipAss}
\ee
for some positive constant $\|\psi\|_{\lip}.$
\item[$(2)$] 
Suppose 
$(f,f') (\cdot \ ;\BX) $ belongs to $ \BD_X^{2\beta } \C^{\gamma}_b$ and $(D_yf,D_yf')(\cdot\ ;\BX)$ belongs to $\BD_X^{2\beta}\C_b^{\gamma-1}$ for $\beta \in (0,\alpha]$, $\gamma > \frac{1}{\alpha}$ with $\gamma\leq3$ such that $\alpha+\beta>\frac12$ and $\alpha + (\gamma-1)\beta>1.$
\item[$(3)$] Suppose $g \in \mathrm{BUC}(\R^{d_Y})$, i.e. bounded uniformly continuous, with concave modulus of continuity $\lambda_g (.)$.
\end{assumption}


\subsection{Regularity of the rough value function}
Under the above assumption, we show the continuity of the rough value function \eqref{vf-rsde11}, which plays a key role in the proof of the dynamical programming principle.
Following Remark \ref{rem:remove_running_cost} we only treat $\ell \equiv 0$, but see Remark \ref{rem:cont_val}, (ii). 
\begin{thm} \label{thm:RoughValueReg}
Suppose $\BX,\bBX  \in \Ca_T$, and 
$(  b,   \sigma,    f  , f',  g)$, 
$(\bar b, \bar \sigma,  \bar f, {\bar f}', \bar g)$ 
satisfy Assumption \ref{assum-dpp}. For any $\ctrl \in \MA,$ let $\MV(t,y;\BX)$ (resp. $J(t,y,\ctrl;\BX)$) and $  \bar{\mathcal{V}} (\bar t,\bar y;{\bar \BX})$ (resp. $\bar{J}(\bt,\by,\ctrl;\bBX)$) be the corresponding value functions (resp. cost function) given by \eqref{vf-rsde11} (resp. \eqref{cost-rp}) with data 
$(b,\sigma, g, f, f' ,\BX)$ and $(\bar b, \bar \sigma, \bar g, \bar f ,{\bar f}', \bar \BX)$ 
respectively. ,
 \begin{multline} \label{localest-costfct}
	 	\sup_{\ctrl \in \MA } |\bar {J}(\bar t,\bar y, \ctrl;\bar \BX)- {J}(t,y, \ctrl;\BX)|
        \lesssim | \bar{g}-g |_\infty
	 	\\
	 	+\lambda_{g} \left(|T-t\vee\bar t|^\alpha(
		|(b,\sigma)-(\bar b, \bar \sigma) |_\infty
		+  F %
		+ \rho_\alpha(\bar\BX,\BX)
		)+|t-\bar t|^{\alpha}+|y-\bar y| \right)
 	\end{multline}
with $F = \|(f-\bar f, f'-\bar f')\|_{\gamma-1}+  \llbracket f, f'; \bar{f}, \bar{f}' \rrbracket_{X, \bar{X}; 2 \beta}$. 
Moreover,
       \begin{multline} \label{equ:cRSDEestimate'}
	 	|\bar {\mathcal{V}}(\bar t,\bar y;\bar \BX)- {\mathcal{V}}(t,y;\BX)|
        \lesssim | \bar{g}-g |_\infty
	 	\\
	 	+\lambda_{g} \left(|T-t\vee\bar t|^\alpha(
		|(b,\sigma)-(\bar b, \bar \sigma) |_\infty+ F + 
		\rho_\alpha(\bar\BX,\BX)
		)+|t-\bar t|^{\alpha}+|y-\bar y| \right),
 	\end{multline}
and the implicit constant depends on the parameter $M$ from \cref{thm:stabilityforRSDEs_compendium}, in particular on $\nn{\bBX}_\alpha+\nn{\BX}_\alpha$.
\end{thm}
 
\begin{remark} \label{rem:cont_val} (i) At first reading, take $f = \bar{f} \in \C^\gamma_b, f' = \bar{f}' \equiv 0$, so that $F = 0$. In this case,  \cref{thm:RoughValueReg} immediately implies that the
rough value function depends continuously on the rough path. 
(ii) In presence of additional running costs $\ell$ and $\bar{\ell}$,  under Assumption \ref{assum-dpp}, applying the previous estimates to the state-augmented problem of Remark \ref{rem:remove_running_cost} easily leads to extensions of the above estimates. In particular, with \(\tilde{\lambda}_g(\cdot)=\lambda_g(\cdot)+|\cdot |\), 
 \begin{multline} \label{equ:cRSDEestimate}
	 	|\bar {\mathcal{V}}(\bar t,\bar y;\bar \BX)- {\mathcal{V}}(t,y;\BX)|
        \lesssim | \bar{g}-g |_\infty
	 	\\
	 	+\tilde\lambda_{g} \left(|T-t\vee\bar t|^\alpha(
		|(b,\sigma,\ell)-(\bar b, \bar \sigma, \bar \ell)|_\infty+ F +
		\rho_\alpha(\bar\BX,\BX)
		)+|t-\bar t|^{\alpha}+|y-\bar y| \right).
 	\end{multline}
\end{remark}  
 
\begin{proof} 
Applying triangle inequality, we have, from the very definition of the value function,
	\begin{align*}
		|\bar {\mathcal{V}}(t,y;\bar{\BX})- {\mathcal{V}}(t,y;\BX)|\le |\bar g-g|_\infty+ \sup_{\ctrl  }  |\E g(Y^{\bar t,\bar y}_T(\ctrl,\bar\BX))-\E g( Y^{ t, y}_T(\ctrl,\BX))|,
	\end{align*}
	where we write $Y^{(t,y,\ctrl,\BX )}_s$ as $Y^{ t, y}_s(\ctrl,\BX).$
	Hence, it suffices to consider the case $\bar g=g$.
	By regularity of $g$ and Jensen inequality, 
	\begin{align*}
		| \E [g( Y^{\bar t,\bar y}_T(\ctrl,\bar{\BX}))]-\E [g( Y^{t,y}_T(\ctrl,\BX))]|
		&\le \E \lambda_g(|Y^{\bar t,\bar y}_T(\ctrl,\bar{\BX})-Y^{t,y}_T(\ctrl,\BX)|)
		\\&\le \lambda_g(\E| Y^{\bar t,\bar y}_T(\ctrl,\bar{\BX})-Y^{t,y}_T(\ctrl,\BX)|).
	\end{align*}	
	We assume that $t\le\bar t$.
	By uniqueness of \eqref{c-rsde}, we can write $Y^{t,y}_T(\ctrl,\BX)=Y^{\bar t,Y^{t,y}_{\bar t}(\ctrl,\BX)}_T(\ctrl,\BX)$ and apply \cref{thm:stabilityforRSDEs_compendium}	
	to get that 
	\begin{align*}
		\E| Y^{\bar t,\bar y}_T(\ctrl,\bar{\BX})-Y^{t,y}_T(\ctrl,\BX)|
		&=\E| Y^{\bar t,\bar y}_T(\ctrl,\bar\BX)-Y^{\bar t,Y^{t,y}_{\bar t}(\ctrl,\BX)}_T(\ctrl,\BX)|
		\\&\le\|Y^{t,y}_{\bar t}(\ctrl,\BX) -\bar y\|_2+\|\delta Y^{\bar t,\bar y}_{\bar t,T}(\ctrl,\bar \BX)-\delta Y^{\bar t,Y^{t,y}_{\bar t}(\ctrl,\BX)}_{\bar t,T}(\ctrl, \BX)\|_2
		\\&\lesssim\|Y^{t,y}_{\bar t}(\ctrl,\BX)-\bar y\|_2+|T-\bar t|^\alpha\rho_\alpha(\bar\BX,\BX). 
	\end{align*}

	Using triangle inequality, the identity $Y^{t,y}_t(\ctrl,\BX)=y$ and estimate \eqref{est.apri.m} (noting that  $p\ge2$ can be chosen arbitrarily), we have	
	\[
		\|Y^{t,y}_{\bar t}(\ctrl,\BX)-\bar y\|_2
		\le \|Y^{t,y}_{\bar t}(\ctrl,\BX)-Y^{t,y}_t(\ctrl,\BX) \|_2+|\bar y-y|
		\lesssim|\bar t-t|^\alpha+|\bar y-y|.
	\]
	The implicit constants in the previous estimates are uniform in $t,y,\bar t,\bar y$ and crucially also in $\ctrl$, thanks to \cref{assum-dpp}.
	Putting these estimates altogether, in conjunction with elementary estimates of the form $| \inf A  - \inf B |  \le \sup | A - B |$, we obtain the announced inequality.
\end{proof}

\subsection{Rough dynamical programming principle}

To show the Dynamical Programming Principle (DPP) in the canonical space, 
we exploit the causal structure of optional processes. 

\begin{lem} \label{lem:op}
 
Let $(\Omega, \cff, (\cff_t)_t, \P)$ be the augmented canonical space as described in \cref{sub.setupdpp}. Suppose that $\ctrl:[0,T] \times \Omega \rightarrow A$ is measurable. Then $\ctrl$ is optional w.r.t. $(\cff_t)$ if and only if for a.e. $\ome,$ $\ctrl_t(\ome)= \ctrl_t(\ome_{t \wedge .})$ for all \(t\).
	
\end{lem}

\begin{proof}
Let $(\MF^B_t)$ be the raw filtration generated by the canonical process $B.$ Then according to \cite[Theorem IV.78 and Remark IV.74]{DM1978}, we see that $(\MF_t)$-predictable processes are indistinguishable from $(\MF^B_t)$-predictable ones.

On the other hand, since $(\MF_t)_t$ is augmented from $(\MF^B_t)_t$, by \cite[p.30 Proposition and p.31 Example]{chung13}, which shows every $(\MF_t)_t$-stopping time\footnote{Note that in that book, stopping times are called optional times.} is a  $(\MF_t)_t$-predictable stopping time (so the optional sigma algebra agrees with the predictable sigma algebra), we see that any optional process $\ctrl(\cdot)$ is indistinguishable from a $(\MF^B_t)$-predictable process. 

Finally, according to \mbox{\cite[p.147 Theorem 97]{DM1978}}, which claims that $\eta$ is $(\MF^B_t)$-predictable if and only if $\ctrl_t(\ome)= \ctrl_t(\ome_{t \wedge .})$, we see our claim follows.
 
\end{proof}


The following lemma, adapted from \cite[Section 3.2]{NT13} and \cite[Proposition 4]{CTT16}, provides convenient presentation of the value function, which is used to prove the ``harder'' part of the DPP. 

\begin{lem}\label{equiv-vf}
	Suppose that Assumption \ref{assum-dpp} holds. Let $t \in [0,T)$, \(y\in \R^{d_Y}\) and \(\BX\in \Ca_T\). Put
\be
\begin{split}
\MA^t&:=\{\ctrl\in\MA:
                 \ctrl|_{[t,T]}\text{ is }(\MF^t_r)\text{-optional}\},
\end{split}
\ee	
\be\label{vf2}
\tilde \MV (t,y; \BX ) := \mathrm{inf}_{\ctrl  \in {\mathcal{A}^t} } J (t,y, \ctrl; \BX ),
\ee
{where $\MF^t_r$ is the usual augmentation of $\sigma(B_u-B_t:t\leq u\leq r)$, and} $J$ is given by \eqref{cost-rp}.
Then we have
\be
\MV(t,y;\BX)=\tilde \MV(t,y;\BX).
\ee	
	
	
\end{lem}

We leave the proof of the above lemma for later and present the DPP for rough stochastic control problem.

\begin{thm}\label{thm-RBP}
(Rough DPP) Let $\BX \in \Ca_T$ and let Assumption \ref{assum-dpp} be in force. 
Then, for  $0 \le s \le s+h \le T$,
\be\label{eq:rdpp}
         \mathcal V (s,y; \BX) =   \inf_{\ctrl  \in \mathcal{A}} \E  \left(  \mathcal V(s+h,Y^{(s,y,\ctrl, \BX)}_{s+h}; \BX) +  \int_s^{s+h} \ell (t, Y_t^{(s,y,\ctrl, \BX)}, \eta_t) d t  \right).
\ee

\end{thm}

\begin{proof}
Denote by $\tMV(s,y;\BX)$ the right hand side of \eqref{eq:rdpp}. We first offer a simpler proof under the assumption (a.1) that the initial datum and coefficients of our controlled RSDE \eqref{c-rsde} do {\em not} depend on $\BX$, which also rules out controlled $(f,f')$. In this case our regularity estimate \eqref{equ:cRSDEestimate} implies that the rough value function depends continuously on $\BX$. Assuming furthermore (a.2) that $\BX$ is a geometric rough path, it is clear that the rough value function is the limit of value functions of classical stochastic control problems. 

\smallskip

(Simplified proof) [Assuming (a.1),(a.2).] This proof uses validity of the 
 (classical) DPP in the case when $\BX$ is replaced by smooth $X$, and uses continuity of $\mathcal{V}=  \mathcal V (s,y; \BX)$ in the rough path argument $\BX$, as follows from  \cref{thm:RoughValueReg} (cf. Remark \ref{rem:cont_val}). 
As before, we may
 assume $\ell = 0$ without loss of generality. 
 By definition, a geometric rough path $\BX$ is the limit of smooth paths.
That is, $\rho_\alpha (\BX,\BX^\varepsilon) \to 0$ as $\varepsilon \to 0$, where $\BX^\varepsilon$ 
is the canonical rough path lift of smooth $X^{\varepsilon}$. 
Set $\mathcal{V}^\varepsilon (s,y) := \mathcal V (s,y; \BX^\varepsilon)$ and similarly for 
$\tilde{\mathcal{V}}^\varepsilon$. 
For fixed $\varepsilon > 0$ we are in a classical setting 
(e.g. \cite{krylov2008controlled, SA21}) 
so that 
$\mathcal{V}^\varepsilon = \tilde{\mathcal{V}}^\varepsilon$, for any $\epsilon > 0$. 
By  \cref{thm:RoughValueReg}, we know $\mathcal{V}^\varepsilon \to \mathcal{V}(\cdot, \cdot; \BX)$ so it suffices to see the corresponding statement for $\tilde{\mathcal{V}}^\varepsilon.$ 
But this is straight-forward, using the elementary $| \inf A  - \inf B |  \le \sup | A - B |$, continuous dependence of $Y(\ctrl, \BX)$ in $\BX$, joint continuity of $V(s,y;\BX)$ in $(y,\BX)$, boundedness of $V$, and bounded convergence. This justifies passage to the limit, and the rough DPP is established.\\

(Full proof) Now fix any $\BX \in \Ca_T$, and we omit \(\BX\) in the notation, writing  $Y^{s,y, \ctrl} = Y^{(s,y,\ctrl,\BX)}$ accordingly. 
This argument does not require continuity of the value function in $\BX$, hence also applicable to suitable $X$-controlled vector fields $(f,f')$. We do need however continuity of the value function in the time-space argument $(s,y)$, as follows immediately from
our estimate \eqref{equ:cRSDEestimate} applied with $t, \bar{t}, y, \bar{y}$, and identical data ($b, \sigma, \ell, g, \BX, f,f'$) otherwise. The following argument is based on the pseudo-Markov property of controlled diffusion (see \cite{CTT16} for details) and the above continuity.

``$\ge$'': By the definition of $\MV$ and Lemma \ref{lem:op}, for any $\vep>0,$ there exists $\ctrl^{\vep}_t(\ome)= \ctrl^{\vep}(t,B(\ome)_{t\wedge .}) \in \MA$ such that 
$$ 
\MV(s,y) > J(s,y,\ctrl^\vep)-\vep.
$$
On the other hand, note that for any $\eta_t(\ome)=\eta(t, B_{t\wedge .}(\ome)) \in \MA,$ we have $ \P$-a.s.
\be\label{eq:r-markov}
 \E \left[ g(Y_T^{s+h, Y_{s+h}^{s,y,\eta}, \eta }) + \int_{s+h}^T \ell(r,Y^{s,y,\eta}_r, \eta_r) dr \Big| \MF_{s+h} \right](\ome) =J(s+h,Y^{s,y,\eta}_{s+h}(\ome), \eta(\ome;\cdot)),
\ee 
where for almost surely $\ome \in \Omega,$ $  \eta_t(\ome;\ome'):= \eta(t, B^{s+h, \ome}_{t \wedge .}(\ome')),\ \text{for any } (t,\ome') \in [0,T] \times \Omega,$ with  
$$
B^{s+h, \ome}_{r}(\ome')=\left\{
\begin{array}{ll}
	 B_{r}(\ome), & r\le s+h,\\
	 B_{r}(\ome')-B_{s+h}(\ome')+ B_{s+h}(\ome), & r >s+h.
\end{array}
\right.
$$    
According to Lemma \ref{lem:op}, we see that $ \eta(\ome;\cdot) \in \MA.$ 
Then by taking $\ctrl=\ctrl^{\vep}$ in \eqref{eq:r-markov} and the tower property of conditional expectation, 
\begin{equation}
	\begin{split}
	J(s,y,\ctrl^\vep) \ge \E \left[ \MV(s+h,Y_T^{s,y,\ctrl^\vep}) + \int_s^{s+h} \ell(r, Y^{s,y,\ctrl^\vep}_r, \ctrl_r)  dr \right],
\end{split}
\end{equation}
which implies $\MV(s,y)> \tMV(s,y)-\vep.$

``$\le$'': For any $\vep >0,$ in view of the uniform continuity of $\MV$ in the space variable, i.e. \eqref{equ:cRSDEestimate}, there exists $\delta>0$ and a sequence of balls $\{B_i \}_{i\ge 1}$ covering $\R^{d_Y}$, such that for any $y,y'\in B_i$ and $t > 0,$
\begin{equation}\label{close-vf}
	|\MV(t,y)-\MV(t,y')|<\vep.
\end{equation}
Moreover, in view of the continuity of $Y^{s,y,\ctrl}$ in $y\in \R^{d_Y}$ uniformly on $\eta \in \MA,$ i.e. Theorem \ref{thm:stabilityforRSDEs_compendium}, we have 
\be
\left\| \sup_{t\in[0,T]} |Y_t^{t,y,\ctrl} - Y_t^{t,y',\ctrl}| \right\|_2 \lesssim |y-y'|,
\ee
where the implicit constant depends only on ${\nn{\BX}}_{\alpha}$ and $(b,\sigma,f,f').$ Thus by Lipschitzness of $(g,\ell),$
we can choose $\{B_i \}_{i\ge 1},$ such that for any $y,y'\in B_i$,
\begin{equation}\label{close-cost}
	 \E\left[|g(Y_T^{t,y,\ctrl})- g(Y_T^{t,y',\ctrl})|+ \int_t^T |\ell(r,Y_r^{t,y,\ctrl}, \ctrl_r)-\ell(r,Y_r^{t,y',\ctrl}, \ctrl_r)|dr  \right] <\vep.
\end{equation}
Then it is standard to obtain a Borel partition of $\R^{d_Y}$, denoted by $ \{\Gamma_i\}_{i\ge 1}$ with some $y^i \in \Gamma_i$ and $\Gamma_i \cap \Gamma_j= \emptyset $ for any $i \neq j$, such that \eqref{close-vf} and \eqref{close-cost} hold for any $y \in \Gamma_i$ and $y'=y^i.$
For any $y^i$, according to Lemma \ref{equiv-vf}, there exists $\eta^{\vep,i} \in \MA^{s+h},$ such that 
\begin{equation}\label{rep-vf}
	\MV(s+h,y^i) \ge J(s+h,y^i, \ctrl^{\vep,i})- \vep.
\end{equation}
Now for any $\eta \in \MA,$ let 
\begin{equation}
	{\bar{\ctrl}}^\vep_t(\ome):=\left\{
	\begin{array}{ll}
	\ctrl^{\vep,i}_t(\ome), & t > s+h, \ Y^{s,y,\ctrl}_{s+h}(\ome) \in \Gamma_i,\\
	\ctrl_t(\ome), & t\in [s,s+h].
\end{array}
	\right.
\end{equation}
According to \eqref{equiv-vf} and \eqref{rep-vf}, we have 
\begin{equation}\label{appro-vf}
	\begin{split}
		\MV(s+h, Y^{s,y,\ctrl}_{s+h}) & \ge \sum_{i} \MV(s+h,y^i) 1_{\{Y^{s,y,\ctrl }_{s+h} \in \Gamma_i \}}- \vep\\
		& \ge  \sum_{i} J(s+h,y^i, \ctrl^{\vep,i}) 1_{\{Y^{s,y,\ctrl }_{s+h} \in \Gamma_i \}}- 2\vep.
	\end{split}
\end{equation}
On the other hand, since $\ctrl^{\vep,i} \in \MA^{s+h}$ is independent of $\MF_{s+h}$ which implies $Y^{s+h,y^i,\ctrl^{\vep,i}}$ is as well, we have
\begin{align*}
	&	\E \left[\sum_{i} J(s+h,y^i, \ctrl^{\vep,i}) 1_{\{Y^{s,y,\ctrl }_{s+h} \in \Gamma_i \}} \right] \\
	 = & \ \E\left[ \sum_i g(Y_{T}^{s+h,y^i,\ctrl^{\vep,i} }) 1_{\{Y^{s,y,\ctrl }_{T} \in \Gamma_i \}} + \sum_i \int_{s+h}^T \ell(r,Y_r^{s+h,y^i, \ctrl^{\vep,i}}, \ctrl^{\vep,i}_r )1_{\{Y^{s,y,\ctrl }_{s+h} \in \Gamma_i \}} dr \right]\\
	 \ge & \ \E \left[ g(Y_T^{s+h,Y^{s,y,\ctrl}, \bctrl^\vep}) + \int_{s+h}^T \ell(r,Y_r^{s+h,Y^{s,y,\ctrl}, \bctrl^\vep}, \bctrl^{\vep}_r ) dr \right] - \vep\\
	  = & \ \E \left[ g(Y_T^{s,y,\bctrl^\vep}) + \int_{s+h}^T \ell(r,Y_r^{s,y,\bctrl^\vep}, \bctrl_r^{\vep}) dr \right]-\vep,
\end{align*}
where we apply \eqref{close-cost} in the last inequality.
Combining with \eqref{appro-vf}, we have
\begin{equation}
\begin{split}
	& 	\E \left[ \MV(s+h,Y^{s,y,\ctrl}_{s+h}) + \int_{s }^{s+h} \ell(r,Y^{s,y,\ctrl}_r, \ctrl_r) dr \right]\\
	& \ge  \ \E \left[ g(Y_T^{s,y,\bctrl^\vep}) +  \int_{s }^{T} \ell(r,Y_r^{s ,y, \bctrl^\vep}, \bctrl^{\vep}_r ) dr \right] - 3 \vep \\
	& \ge \MV(s,y)-3 \vep, 
	\end{split}
\end{equation}
which implies $\tMV(s,y) \ge \MV(s,y)-3\vep$ and so our claim follows.
\end{proof}

\begin{proof}[Proof of Lemma \ref{equiv-vf}] These ideas being standard, we only provide a sketch. (Also see \cite[Proposition 4]{CTT16} for another proof, 
   noting that the authors deal with predictable controls, which is equivalent to our setting, as pointed out in the proof of Lemma \ref{lem:op}.) 
Fix $(s,y) \in [0,T] \times \R^{d_Y}$ and omit \(\BX\) in notations. 
We only need to show $\MV(s,y) \ge \tMV(s,y) .$ Indeed, according to Lemma \ref{lem:op}, any $\ctrl \in \MA$ is of the form $\ctrl= \ctrl_t( \ome_{t \wedge .} )$ and hence for any $t\ge s,$ we can write 
$
		\ctrl_t(\ome)= \Phi(t, \ome_{s \wedge .  } , (\ome_{r} - \ome_s )_{s\le r \le t} ),	 
$
for a suitable $\Phi$. By independence of Brownian increments, for any fixed $\ome_{s\wedge .} \in \Omega |_{[0,s]},$ $\ctrl^{s,\ome}_t:=\Phi(t, \ome_{s \wedge .  }, \cdot-\ome_s): \Omega_{[s,t]} \rightarrow U$ is {optional for the shifted Brownian filtration and belongs to $\MA^s$}.
Then by the Fubini theorem, we have
\begin{equation}
	\begin{split}
		J(s,y,\ctrl) & = \int \E \Big[ g(Y_T^{s,y,\ctrl^{s,\ome}}) + \int_s^T \ell(r,Y_r^{s,y,\ctrl^{s,\ome}}, \ctrl^{s,\ome}_r ) \Big] d \P(\ome_{s\wedge .})\\
		& \ge \int \tMV(s,y)  d \P(\ome_{s\wedge .})= \tMV(s,y),
	\end{split}
\end{equation}
which implies the claim. 
\end{proof} 

We remark that our proof for the DPP,  in particular the first variant, extends in a straight-forward way to the case of 
 two-player, zero-sum stochastic differential games \cite{FS89}.

\subsection{Explicit example} The following example is inspired by \cite{BM07} where the authors consider a linear, scalar situation ($d_X = d_Y =1$),
which allows for pathwise solutions. We here consider general dimensions with linear $f: \R^{d_Y} \rightarrow \mathrm{Lin}(\R^{d_X}, \R^{d_Y})$, zero It\^o vector fields $b,\sigma$, so that \eqref{rand-rsde} is a genuine controlled rough differential equation, of the form
$$
\begin{array}{ll}
  & dY^{\eta}_t = \eta_t dt + f (Y^{\eta}_t) d\BX_t, \quad t \in [s,
  T],\\
  & Y_s^{\eta} = y.
\end{array}
$$
Call $P^{\BX}_{t \leftarrow s}: \R^{d_Y} \to \R^{d_Y}$ the (linear) solution flow to the RDE solution flow of
this equation in uncontrolled case ($\eta \equiv 0$). The controlled process
then can be given in mild formulation,
$$   Y_T^{(s, y, \eta , \BX)} = P^{\BX}_{T \leftarrow s} 
  \hspace{0.17em} y + \int_s^T P^{\BX}_{T \leftarrow r} \eta_r dr, \qquad Y_T^{(0, 0 , \eta , \BX)} = \int_0^T P^{\BX}_{T \leftarrow
   r} \eta_r dr.
$$
Assume now that $| \eta | \leqslant 1$, fix $v \in \mathbb{R}^{d_Y}$ and
consider distance-to-$1$ cost $| \Pi_T^{\eta} - 1|$, where
\[ {\Pi_T^{\eta}}  : = \langle v, Y_T^{ (0, 0 , \eta , \BX) } \rangle =
   \int_0^T \langle \Theta_r, \eta_r \rangle dr, \qquad \Theta_r : =
   (P^{\BX}_{T \leftarrow r})^{\star} v .\]
Clearly ${\Pi_T^{\eta}}  \leqslant M_T^{\BX} \assign \int_0^T |
\Theta_r | d r$ with equality when $\eta_r = \hat{\Theta}_r \assign \Theta_r /
| \Theta_r |$. If $M_T^{\BX} > 1$ we overshoot the target $1$ and the
optimal control (zero cost) is plainly given by $\eta^{\star}_r =
\hat{\Theta}_r / M_T^{\BX}$, thanks to linearity of ${\Pi_T^{\eta}} $
in $\eta$. If $M_T^{\BX} \leqslant 1$ the best we can is full speed
ahead, i.e. $\eta^{\star}_r = \hat{\Theta}_r$ which results in a cost $(1 -
M_T^{\BX})$.

Summarizing, we have
\[ \mathcal{V} (0, 0, \BX) \assign \inf_{\eta \in \mathcal{A}} |
   \langle v, Y_T^{(0, 0 , \eta , \BX)} \rangle - 1| = [1 -
   M_T^{\BX}]^+ \]
with optimal control
\[ \gamma_t^{\star} = \left\{ \begin{array}{ll}
     \hat{\Theta}_r / M_T^{\BX} & \text{on } \{ M_T^{\BX} >
     1 \},\\
     \hat{\Theta}_r & \text{on } \{ M_T^{\BX} \leq 1 \} .
   \end{array} \right. \]
   
Remark that $\Ca_T \ni \BX \mapsto M^{\mathbf{X}}$ is $\BCa_T$-measurable, which (trivially) implies that the process $ \gamma^{\star}$ is $\BCa_T$-optional, but certainly not causal in $\BX$.

\subsection{Joint measurability of \texorpdfstring{$\eps$}{epsilon}-minimizers}

To conclude the current section, we provide a result on the $\vep$-approximation of the optimal control to the rough stochastic control problem, which turns out to be quite helpful for the pathwise control problem considered in  \cref{sec:pathwise-control}. As before, we assume without loss of generality zero running cost $\ell$.


\begin{lem} \label{lem1} 
Let $\BX \in \Coa$, and
suppose Assumption \ref{assum-dpp} holds. Recall that $\mathcal{A}$ is the class of optional (in sense of $\ooo / \aaa$-measurable) controls and that
  \[   
  \mathcal V (s,y; \BX ) = \mathrm{inf}_{\ctrl  \in {\mathcal{A}}} \E\left[ g(Y_T^{(s,y,\ctrl, \BX)}) \right] 
  \]
where $Y^{(s,y,\ctrl,\BX)}$ is the solution to \eqref{c-rsde}. Then for every $\varepsilon > 0$, there exists a 
$\BCoa_T$-optional map $\hat{\eta} = \{
  \eta^{\varepsilon}_t (\ome, \BX) : 0 \leqslant t \leqslant T \}$ such
  that
  \[ 
 \mathcal V (s,y; \BX )) \leqslant \mathbb{E} [g (Y^{(s,y,\hat \ctrl,\BX)}_T)] \leqslant \mathcal V (s,y; \BX ) + \varepsilon.
   \]
    As a consequence, $\mathcal{V}$ is also the value function when $\mathcal{A}$ is replaced by the class of $\BCoa_T$-optional controls. 
\end{lem}
\begin{rem}\label{g-dep-X}
If $g$ is bounded and continuous, so is the value function (w.l.o.g. with $\ell \equiv 0$) from the definition of $\mathcal{V}$. If $\mathcal{V}(s,y;\BX)$ is furthermore continuous in $\BX$, the above lemma applies directly with 
 $g (Y^{(s,y,\hat \ctrl,\BX)}_T)$ 
 replaced by $\mathcal V(s+h,Y^{(s,y,\ctrl, \BX)}_{s+h}; \BX)$, term seen on the right-hand side of \eqref{eq:rdpp}. 
\end{rem} 

\begin{proof}
As heuristic motivation, assume $\Coa_T$ is replaced by $\{ \mathbf{x}_1,
  \ldots ., \mathbf{x}_n \}$, with power $\sigma$-field ${\mathfrak{P}=
  2^{\{ \mathbf{x}_1, \ldots ., \mathbf{x}_n \}}}  .$ Let
  $\eta^{\varepsilon}_t (\ome, \mathbf{x }_i)$ be an $\ooo $-measurable $\varepsilon$-optimal control, for fixed $\mathbf{x }_i$.
  Then
  \[ \eta^{\varepsilon}_t (\ome, \BX) \assign \eta^{\varepsilon}_t (\ome,
     \mathbf{x }_i) \quad \text{when } \mathbf{X} = \mathbf{x}_i, 1
     \leqslant i \leqslant n, \]
  which is clearly $(\ooo \otimes \mathfrak{P})$-measurable.

 In the general case, employ continuity  (at $\mathbf{X}$): by \mbox{\cref{thm:RoughValueReg}}, for every
  $\varepsilon, \mathbf{X}$ have $\delta(\varepsilon, \mathbf{X}) >0$ such that for any $\bx\in \MC_T,$
  \[ d (\mathbf{x}, \mathbf{X }) < \delta (\varepsilon, \mathbf{X})\Rightarrow \sup_{\eta \in
     \mathcal{A}} \left(\mathbb{E} [g (Y^{\eta, \mathbf{X}}_T)] -\mathbb{E} [g
     (Y^{\eta, \mathbf{x}}_T)]\right) < \varepsilon . 
     \]
  Consider the open cover of $\Coa_T$ given by $\{ B (\mathbf{X},
  \delta (\mathbf{X}, \varepsilon)) : \mathbf{X} \in \Coa_T \}$. \
  By the Lindeloef property of Polish spaces, one may switch to a countable
  subcover, for some $\{\mathbf{x}_i\}_i \subset \Coa_T,$
  \[ B_i : = \{ B (\mathbf{x}_i, \delta (\varepsilon, \mathbf{x}_i)) : i
     \in \mathbb{N} \} . \quad \]
  We can turn this into a (Borel) measurable partition $\{P_i\}_i$ of $\Coa_T$
  by considering $P_1 = B_1, P_2 = B_2 \backslash B_1, \ldots .$, so that every $\mathbf{X} \in P_{\mathfrak{i} (\mathbf{X})}$ for some
  unique $\mathfrak{i} (\mathbf{X}) \in \mathbb{N}$. We can now define
\[ \eta^{\varepsilon}_t (\ome, \mathbf{X}) \assign \eta^{\varepsilon}_t
     (\ome, \mathbf{x }_{\mathfrak{i} (\mathbf{X})}), \quad \tmop{when }\ 
     \mathbf{X} \in P_{\mathfrak{i} (\mathbf{X})}, 
\]
  which is clearly $(\ooo \otimes \BCoa_T)$-measurable. 
\end{proof}
 
\section{Rough stability of Hamilton--Jacobi--Bellman, revisited}\label{sec:HJB}

Let $\MV(s,y;\BX)$ be given by \eqref{c-rsde} and \eqref{vf-rsde11}.

%
If the (geometric) rough path $\BX$ is replaced by a continuously differentiable path $X$, so that $d X = \dot{X} d t$, it is
classical that the value function is the unique (in viscosity sense; see e.g. \cite{fleming2006controlled}) solution
to the HJB terminal value problem
\[ - \partial_t \MV = H (y, t, D \MV, D^2 \MV) + (f (t,y) \cdot D \MV) \dot{X}, \qquad \MV(T,\cdot) \equiv g, \]
with
\[ H (y, t, D \MV, D^2 \MV) = \inf_{\ctrl \in A} \left( \tfrac{1}{2} \mathrm{Tr} \left(
   (\sigma \sigma^T) \left(t, y, \ctrl \right) D^2 \MV \right) + b \left( t, y,
   \ctrl \right) \cdot D \MV \right). \]
In the general case, one expects a ``rough'' HJB equation of the form,
$$
 - d_t \MV = H (y, t, D \MV, D^2 \MV) dt  + (f (t,y) \cdot D \MV ) d \BX.
$$
The pragmatic interpretation of such an equation, point of view also taken in \cite{Friz2016}, is to view $\MV$ as limit of $\MV^\eps$, whenever $X^\eps \to \BX$ in rough path sense, provided $\MV$ only depends on $\BX$ and not on the approximating sequence. Such a construction was carried out in \cite{CFO11}, in case of autonomous coefficients fields, relying fundamentally on representing $\MV$ via a ``rough flow'' transformation, induced by the auxiliary RDE $- d \Psi= f (\Psi) d \BX$; the (formally) transformed equation 
can also serve as a definition, as pointed out in \cite[Ch.13]{FH20} and references therein.\footnote{Most authors using flow transformation techniques consider autonomous coefficients, i.e. without $t$-dependence, but the extension is not difficult, provided one works with the correct backward flows.}  Equivalently, one formulates the transformation via the transport equation associated to the RDE, see \cite[Ch.12]{FH20} and \cite{CHT24}. See also \cite{BCO23} for a discussion how this relates to the original ``local'' transform of test functions.  

For the reader's convenience we restate the relevant result in our situation.
\begin{thm}[{\cite{CFO11}}]\label{thm51} Let $\BX \in \Coa_{g,T}$. Assume $b=b(y,\eta), \sigma=\sigma(y,\eta)$ are bounded Lipschitz in $y$, uniformly in $\eta$. Assume further  $f \in \C_b^{\gamma+2}$, with $\gamma > 1/\alpha$. Then the solution map $(\BX, g) \mapsto \MV$ is continuous, seen as map
from $\Coa_{g,T} {\times} \mathrm{BUC}(\R^{d_Y}) \mapsto \C_b([0,T]\times \R^{d_Y})$.
\end{thm} 
\begin{proof}  This is \cite[Theorem 1]{CFO11}, noting that $f \in \C_b^{\gamma+2}$ implies the required $\Phi^{(3)}$-invariant comparison put forward in  \cite{CFO11}.
\end{proof}

We can now improve on this result in several ways.

\begin{thm} \label{thm52} Assume $b=b(t,y,\eta), \sigma=\sigma(t,y,\eta)$ are bounded Lipschitz in $y$, uniformly in the $t, \eta$. 
Assume further  $f \in \C_b^{\gamma}$, with $\gamma > 1/\alpha$. Then the solution map $(\BX, g) \mapsto \MV$ is continuous, seen as map
from $\MC^\alpha_T {\times} \mathrm{BUC}(\R^{d_Y}) \mapsto \mathrm{BUC}([0,T]\times \R^{d_Y})$ and this continuity statement is fully quantifiable via estimate \eqref{equ:cRSDEestimate}. In particular, the map $\BX \mapsto \MV$ is locally Lipschitz continuous.
\end{thm} 


\begin{proof} Direct consequence of  \cref{thm:RoughValueReg}.
\end{proof} 

In particular, in the case of interest when $\alpha = 1/2 - \eps$ (which covers the case when $\BX$ is a typical realisation of Brownian motion enhanced to a rough path in the Stratonovich sense) the excessive regularity requirement $f \in \C_b^{4+\eps}$ of Theorem \ref{thm51} is replaced by $f \in \C_b^{2+\eps}$ in  Theorem \ref{thm52}, which is the natural condition (with $\eps = 0+$) one expects from Stratonovich SDE and transport SPDE theory. 

\begin{exam} We revisit \cite[Ex 3.17]{AC20}. The interest in doing so comes from the $(...)dB$ term which fits directly in our setup, while outside the scope of that work.\footnote{From \cite{AC20}: ``we expect all of the analysis [with $(...)dB$] to follow with appropriate technical adjustments.''} Consider the following controlled rough SDEs.
\begin{align*}
   d \left(\begin{array}{c}
      Y_t^{s, y, z, \eta, \BX }\\
      Z_t^{s, z, \eta, \BX }
    \end{array}\right) &= \left(\begin{array}{c}
      0\\
      \eta_t
    \end{array}\right) d t + \left(\begin{array}{c}
      Z_t^{s, z, \eta, \BX}\\
      0
    \end{array}\right) \sigma d B_t + \left(\begin{array}{c}
      Z_t^{s, z, \eta, \BX_t}\\
      0
    \end{array}\right) d \BX, 
    \\ \left(\begin{array}{c}
      Y_s^{s, y, z, \eta, \BX }\\
      Z_s^{s, z, \eta, \BX }
    \end{array}\right) &= \left(\begin{array}{c}
      y\\
      z
    \end{array}\right) ,
\end{align*}
where $\BX$ is a (scalar) rough path that we assume geometric. 
It is
implicit in this rough SDE that $Z ^{s, z, \eta, \BX}$ has zero
stochastic Gubinelli derivative, the stochastic rough integral $\int Z d
\BX$ then makes sense without the compensating terms (of the form
$\sum Z'_u \mathbb{X}_{u, v}$) i.e. converges directly as Riemann-Stieltjes
integral (so that this particular example can also be understood pathwise,
via integration by parts, exploiting finite variation of $Z ^{s, z, \eta,
\BX}$ ). 
Consider the following value function for the rough stochastic control problem
\[ \MV (s, y, z) = \sup_{\eta} \mathbb{E} \left[ Y_T^{s, y, z, \eta; \BX} - \int_s^T
   \varepsilon \eta_t^2  \hspace{0.17em} dt \right] . \]
An easy localisation argument shows that this example fits in the setting of
Theorem \ref{thm52}. Replacing the rough path $\BX=(X,(\delta X)^2/2)$ by a smooth path $X$, the value function satisfies
\[ - d_t \MV = \frac{1}{2} z^2 \sigma^2  \frac{\partial^2 \MV}{\partial y^2} 
   \hspace{0.17em} dt + \frac{1}{4 \epsilon} \left( \frac{\partial \MV}{\partial
   z} \right)^2  \hspace{0.17em} dt + z \frac{\partial \MV }{\partial x} 
   \hspace{0.17em} d X, \qquad \MV (T, y, z) = y, \]
   the explicit solution of which is 
   $$ \MV (s, y, z) = y + (X_T - X_s) z + \frac{1}{4 \varepsilon} \int_s^T (X_T -
X_s)^2 d s.$$
In fact, this form of the value function remains valid in the rough case. Since  $\BX \in \mathscr{C}^{0,\alpha}_{g,T}$ is the limit of canonically lifted $X^\epsilon$, and the (explicit) value function $\MV=\MV^X$ is (obviously)  robust in $X$ (in the sense of continuity w.r.t. sup-topology on $[0,T]$), we see that $\MV^{X^\epsilon} \to \MV^X$, thereby illustrating the abstract stability result of Theorem \ref{thm52}. (This is also verifiable, of course, from the stochastic representation of $\MV$.)
\end{exam}

\section{Pathwise stochastic control with general rough path noise}\label{sec:pathwise-control}

\subsection{Pathwise stochastic control via randomised RSDEs}\label{sec-randomised-rsde}

Consider 
\begin{equation}  \label{sec51productbasis}
	 	( {\Omega},\MG,(\MF_t)_t,\P)  = (\Omega',\mathcal G',(\cff'_t)_t,\P') \otimes (\Omega'',\mathcal G'',(\cff''_t)_t, \P''),
\end{equation} 
where $(\Omega',\mathcal G',(\MF'_t)_t,\P')$ is the canonical space with $B$ a $d_B$-dimensional Brownian motion, $\MF'$ the augmented filtration of $\MF^B$, and $\P'$ the Wiener measure. Write $\omega = (\omega',\omega'')$.
Consider also the rough-path-valued random variable\footnote{The actual filtration on $\Omega''$ and the ambient $\sigma$-field $\mathcal G''$ will play little role in this section. Of course,
one can also work with the filtration generated by $\BW$ and also set $\mathcal G'' = \MF_T''$, if so desired.}
$$\BW: (\Omega'',\MF_T'' ,\P'') \rightarrow  (\Coa_T ,\BCoa_T )$$
and controls $\eta = \eta_t (\ome' , \mathbf{X})$, for $\mathbf{X} \in \Coa_T$, taken in
\begin{equation}
\label{equ:A1}
\mathcal{A}^1 : = \left\{ \eta \text{ is } \mathfrak{O}' \otimes \{
   \emptyset, \mathscr{C}_T \}/ \aaa  \text{-measurable}  \right\} 
\end{equation}
or the larger class   
\begin{equation}
\label{equ:A2} 
   \mathcal{A}^2 : = \left\{ \eta \text{ is } \mathfrak{O}' \otimes
   \mathfrak{C}_T / \aaa \text{-measurable}  \right\}.
   \end{equation}
(Or anywhere ``in-between'', such as the space of causal controls 
$\MA^{\mathrm{causal}}$ defined in \eqref{equ:Acausal} below.) Also write $\bar{\eta} = \bar{\eta}_t (\omega) \in \bar{\mathcal{A}}^i$ if
$\bar{\eta}_t (\omega) = \eta_t (\ome' , \mathbf{W} (\omega'')  )$
for $\eta \in \mathcal{A}^i$, $i = 1, 2.$ Note $\bar{\eta}_t \in \cff'_t
\vee \cff''_T$.
To simplify the discussion, for the remainder of this section we assume running cost $\ell=0$ and $f=f(t,y)$ is sufficiently regular in $t$ so that $f'=0$).
More precisely, we make the following

\begin{assumption} \label{assum-path-stoch}
Let Assumption \ref{assum-dpp} be in force. 
Impose furthermore that $f$ in Assumption \ref{assum-dpp} (2) is independent of $\BX$ and sufficiently regular in the time variable. Specifically, we assume that \(f(t,.)\in \C^\gamma_b\) uniformly in \(t \in [0,T]\), and
	\[ \sup_{(s,t)\in \Delta_T}\frac{| f (t,.) - f (s,.) |_{\infty}} { | t - s |^{2 \beta}}<\infty \]

\end{assumption}
Note that for any $f$ as described in Assumption \ref{assum-path-stoch}, one has $(f, 0) \in \mathbf{D}^{2\beta}_X \mathcal{C}_b^{\gamma}$ for any $\alpha$-H\"older path $X$. Also note that autonomous vector fields $f \in \mathcal{C}^\gamma_b$ automatically satisfy these requirements.

Theorem \ref{thm:wellposed} shows that for $\eta \in \MA^2$, we have a unique solution $Y^{\ctrl, \BX}(\ome')$ to the RSDE
\be\label{r-dsde'}\left\{
\begin{aligned}
 & dY^{\ctrl,\BX}_t(\omega')=b(t,Y^{ \ctrl,\BX}_t ,\ctrl_t(\ome',\BX)  )dt+\sigma(t,Y^{ \ctrl,\BX}_t ,\ctrl_t(\ome',\BX) )dB_t(\omega')+f(t,Y_t^{\ctrl,\BX}) d\BX_t, \\
& Y_s = y,
 \end{aligned}\right.
\ee
of which, under the given assumptions we can and will take a $\BCoa_T$-optional version, cf. Theorem \ref{thm-rsde-optional} (measurable selection). We then have that
\be\label{randomised-rsde}
\bar{Y}^{\bar{\eta}} (\omega) := Y^{\bar{\eta},
\BW (\omega'')} (\omega'):= Y^{\eta, \mathbf{X}}|_{\BX=\BW(\ome'')}
\ee
defines 
a measurable process (see Section \ref{sec:BMcase} for a discussion when it coincides with a solution to a classical SDE) and more precisely
\[ \bar{Y}^{\bar{\eta}}_t (\omega', \omega'') \in \cff'_t \vee
   \cff''_T . \]
Define 
\[ \mathcal{V}^i (s, y ;   \omega) \assign
    {\tmop{essinf}_{\eta  \in \mathcal{A}^i}} \mathbb{E}^{s,
   y} (g (\bar{Y}^{\bar{\eta}}_T)   {| \cff_T^{\mathbf{W}}
    }) =  {\tmop{essinf}_{\bar{\eta}   \in
   \bar{\mathcal{A}}^i}} \mathbb{E}^{s, y} (g (\bar{Y}^{\bar{\eta}}_T)
     | \cff_T^{\mathbf{W}}  ), \]
for $i = 1, 2$, where $\cff_T^{\mathbf{W}} $ is the $\sigma$-algebra generated by $\BW.$ 

%

On the other hand, to apply our dynamical programming for RSDEs to the random case, let 
\be\label{mv'}
\bar{\MV}(s,y; \ome)
:= \MV(s,y;\BW(\ome)) \equiv \mathcal{V} (s, y ;  
   \mathbf{X}) |_{\mathbf{X} = \mathbf{W} (\omega)}
   = \Big(\mathrm{inf}_{\ctrl \in \MA} \E' [g(Y^{\ctrl,\BX})]\Big)\Big|_{\BX=\BW(\ome)},
\ee 
where $\MA$ is introduced in \eqref{def:ad-cont} with $(\Omega, \MO)$ there replaced by $(\Omega', \MO')$, i.e.
\be\label{rough-ad-contr}
\MA:=\{\ctrl: ([0,T]\times \Omega', \MO') \rightarrow (A,\aaa) \text{ is measurable} \}.
\ee

To build the relation between the different random value functions we need the following proposition.

\begin{prop}\label{equ-condi-ep} 
Suppose that $(b,\sigma,f)$ satisfies  \cref{assum-path-stoch} $(1), (2)$. For any $\ctrl  \in   \MA^2,$ let $Y^{\ctrl,\BX}$ be the solution to \eqref{r-dsde'} with $\BX \in \Coa_T$, and $\bar{Y}^{\bar{\eta}}$  be
the randomised solution given by \eqref{randomised-rsde}. Let $g$ be bounded and measurable. Then we have 
 \be \label{equiv-condi}
  \mathbb{E} [g (\bar{Y}^{\bar{\eta}}_T) |
     {\text{$\cff$}}_T^{\mathbf{W}} ] = \mathbb{E} [g (Y^{\eta, \mathbf{X}}_T)] \Big|_{\mathbf{X} = \mathbf{W}}. 
 \ee

\end{prop}

\begin{proof} 
	
According to Theorem \ref{thm-rsde-optional}, we see that $\bar{Y}^{\bar{\eta}}_T$ is measurable on $(\Omega' \times \Omega'', \cff'_T \otimes \cff''_T),$ and thus the left hand side of \eqref{equiv-condi} is well-defined. For any bounded measurable function $\phi$ on $(\Coa_T, \BCoa_T ),$ let $\xi:=\phi(\BW).$ Since $Y^{\eta, \mathbf{X}}_T$ is jointly measurable on $\Omega' \times \Coa_T,$ by the Fubini Theorem and identity \eqref{randomised-rsde}, we have for any $\ome'' \in \Omega'',$ 
  $$
  \mathbb{E}' [g (Y^{\eta, \mathbf{X}}_T)]  |_{\mathbf{X} = \mathbf{W}
     (\omega'')}= \mathbb{E}' [g (Y^{\eta, \mathbf{X}}_T) |_{\mathbf{X} = \mathbf{W}
     (\omega'')}  ]= \mathbb{E}' [g ( \bar{Y}^{\bar{\eta}}_T )  ](\ome'') .
  $$ 
  It follows that 
  $$
  \E \left[ \xi \mathbb{E}' [g (Y^{\eta, \mathbf{X}}_T)]  |_{\mathbf{X} = \mathbf{W}
     }  \right]= \E[\xi g(\bar{Y}^{\bar \eta }_T)],
  $$ 
which implies \eqref{equiv-condi}.
\end{proof}

\begin{rem} \label{cost-g-X}
Similar to Remark \ref{g-dep-X},
 if $\mathcal{V}:[0,T] \times \R^{d_Y} \times \Coa_T \rightarrow \R$ bounded, jointly measurable, the above proposition 
applies directly with $g (Y^{ {\eta},
     \mathbf{X}}_T)$ replaced by $\mathcal{V}(r,Y^{(s,y,\ctrl, \BX)}_{r}; \BX)$ for any $r\in[s,T]$, a term seen on the right-hand side of \eqref{DPP-path-stoch}. 

\end{rem}

%
%
%
%

Then we have the following equivalence between $\bMV$ and $\MV^i$, $i=1,2.$

\begin{theorem}\label{rep-value-fct}
 Suppose that  \cref{assum-path-stoch} holds. Then $ \bar{\mathcal{V}}$ given by \eqref{mv'} is a continuous modification of both $\mathcal{V}^1$ and $\mathcal{V}^2$ with a.s. regularity inherited from the regularity of $\mathcal{V}$ provided by  \cref{thm:RoughValueReg}. In particular, 
for every $(s,y) \in [0,T] \times \R^{d_Y}$, with probability one,
 $$
\bar{\mathcal{V}} (s, y ;  \omega) = \mathcal{V}^1 (s,
  y ;  \omega) =\mathcal{V}^2 (s, y ;  \omega).$$
\end{theorem}
\begin{remark} In \cite[Thm 5.5]{BM07} the authors consider a stochastic value function, that can be seen to coincide with ours in case of Brownian (rough path) randomisation, and show the existence of a continuous version. Our improvement here is two-fold. (1) At no point have we assumed Brownian (rough path) statistics of $\mathbf{W}$, as one might expect since the entire analysis is conditional on the information provided by $\mathbf{W}$. (2)  \cref{thm:RoughValueReg} provides a modulus of regularity for $\mathcal{V}$ which is inherited by $\bar{\mathcal{V}}$ and thus yields regularity results for the stochastic value function, which is new even in the case previously considered by \cite{BM07}.
\end{remark} 

%

\begin{proof}[Proof of \cref{rep-value-fct}] 
Note that $\MA^1 \subseteq   \MA^2 ,$ and thus $\MV^1 \ge \MV^2$ a.s.. Hence we only need to show that $\bMV \ge \MV^1$ and $\MV^2 \ge \bMV$ a.s..
We omit \((s,y)\) in the notation in the remaining of the proof below.

  (i) ``$\MV^2 \ge \bMV$'': Let $0 \leqslant$ $\phi \in \C_b$($\mathscr{C  }_T$). For
  any $u = u_t (\omega' , \mathbf{X}) \in \mathcal{A}^2$, note that for $\mathbf{X}$ fixed, $(t, \omega')
  \mapsto u_t (\omega' , \mathbf{X})$ is in $\mathcal{A}$.
  We have
  \[
   \begin{split}
   	\mathbb{E} [\phi (\mathbf{W} (\ome'')) \MV ( 
    \mathbf{W} (\omega''))]    = & \ \mathbb{E} [\phi (\mathbf{W} (\ome'')) (\inf_{\eta \in \mathcal{A}}
    \mathbb{E}' [g (Y^{\eta, \mathbf{X}}_T)]) |_{\mathbf{X} = \mathbf{W}
    (\omega'')}]\\
  \le & \ \mathbb{E} [\phi (\mathbf{W} (\ome'')) (\mathbb{E}' [g
    (Y^{u , \mathbf{X}}_T)])
    |_{\mathbf{X} = \mathbf{W} (\omega'')}].   	
   \end{split}
  \]
In view of Proposition \ref{equ-condi-ep}, we have 
\[
\begin{split}
\mathbb{E} [\phi (\mathbf{W} (\ome'')) (\mathbb{E}' [g
    (Y^{u, \mathbf{X}}_T)])
    |_{\mathbf{X} = \mathbf{W} (\omega'')}] & =  \mathbb{E} [\phi (\mathbf{W} (\ome'')) \mathbb{E} [g
    (\bY^{ \bar{u}   }_T) \mid
    \cff^{\mathbf{W}}_T]  ],\\
\end{split}
\]
which implies that
  \[ 
 \bMV= \mathcal{V} (  \mathbf{W} (\omega'')) \leqslant
     \tmop{essinf}_{ {u} \in \mathcal{A}^2 } \mathbb{E} [g
     (\bY^{ \bar u }_T) \mid
     \cff^{\mathbf{W}}_T] (\omega'')=\MV^2.
      \]
      
  (ii) ``$\bMV \ge \MV^1$'': For every
  $\varepsilon >0, \mathbf{X} \in \Coa_T,$ by  \cref{thm:RoughValueReg}, there exists $\delta=\delta(\vep,\BX) >0$ such that for any $\bx\in \Coa_T,$
  \[ d (\mathbf{x}, \mathbf{X }) < \delta \Rightarrow \sup_{\eta \in
     \mathcal{A} } \left(\mathbb{E}' [g (Y^{\eta, \mathbf{X}}_T)] -\mathbb{E}' [g
     (Y^{\eta, \mathbf{x}}_T)]\right) < \varepsilon, 
     \]
which implies 
\be\label{vep-appr}
| \MV(\BX)-\MV(\Bx) | < \vep,  \ \ \ |J(\ctrl;\BX)- J(\ctrl;\Bx)|< \vep,
\ee
where we recall $J(\ctrl;\BX):= \mathbb{E}' [g (Y^{\eta, \mathbf{X}}_T)].$ Then by the same argument in the proof of Lemma \ref{lem1}, there exists $\{\mathbf{x}_i\}_i \subset \mathscr{C}_T,$ and 
   \[ B_i : = \{ B (\mathbf{x}_i, \delta (\varepsilon, \mathbf{x}_i)) : i
     \in \mathbb{N} \},  
     \]
 such that $\{B_i\}_i$ is a countable open cover of $\Coa_T$. Let  $\{P_i\}_i$ be a (Borel) measurable partition  of $\mathscr{C}_T$ by $P_1 := B_1, P_2 := B_2 \backslash B_1, \ldots ,$ with some $  \By_i  \in P_i$ for any $i$. 
   
 For any $i,$ let $\ctrl^{i} \in \MA$ be $\vep$-optimal for $\MV(\By_i ),$ i.e.
$$
\MV(\By_i ) \le  J(\ctrl^{i};\By_i ) \le \MV(\By_i ) + \vep.
$$
By the definition of $P_i$, \eqref{vep-appr} and the above inequality, we have
\begin{align} \label{est-3}
	\MV(\BX)|_{\BX=\BW(\ome'') } & = \sum_{i\in \N} \MV(\BW (\ome'') ) 1_{P_i}(\BW(\ome'')) \\ \nonumber
	& \ge \sum_{i\in \N } (\MV(\By_i ) - \vep ) 1_{P_i}(\BW (\ome'')) \\
	\nonumber
	& \ge \sum_{i\in \N } (J(\ctrl^i; \By_{i}) - 2\vep ) 1_{P_i}(\BW (\ome'')).
\end{align}
In view of the second inequality in \eqref{vep-appr}, for any $i,$
\be\label{est-4}
(J(\ctrl^i; \By_i ) - 2\vep ) 1_{P_i}(\BW (\ome''))   \ge (J(\eta^i; \BW (\ome'') )-3 \vep )  1_{P_i}(\BW (\ome'')).
\ee
Note that $u^{i}(\ome',\BX) := \ctrl^{i}(\ome') \in  \MA^1 \subseteq \MA^2 ,$ 
which by  \cref{equ-condi-ep}, implies that 
\be \label{rela-5}
J(\eta^{i}; \BW (\ome'')) = \E'[g(Y_T^{\ctrl^{i},\BX })]|_{\BX=\BW (\ome'')}= \E[ g(\bY^{\bar{u}^i }_T ) | \MF^{\BW}_T ](\ome'').
\ee
Take \eqref{est-4}, \eqref{rela-5} to \eqref{est-3}, and we have 
\begin{align*}
\MV(\BX)|_{\BX=\BW(\ome'') } & \ge \sum_{i\in \N } \E[ g( \bY^{\bar{u}^i }_T ) | \MF^{\BW}_T ](\ome'')   1_{P_i}(\BW(\ome'')) - 3\vep  \\
& \ge \sum_{i\in \N } \MV^1(\ome'') 1_{P_i}(\BW (\ome''))  - 3\vep \\
& = \MV^1(\ome'') -3 \vep.
\end{align*}
Thus upon $\varepsilon \downarrow 0$, we see $\mathcal{V} (  \mathbf{W} (\omega'')) \ge \MV^1 (\omega'')$ a.s..
\end{proof}

\subsection{Stochastic dynamical programming principle} 

Recall $\bY^{s,y,\bctrl}(\ome)$ and $ \bar{\mathcal{V}}(s,y;\ome)$, as previously given in \eqref{randomised-rsde} and \eqref{mv'} respectively. In view of  \cref{thm:RoughValueReg} and \eqref{mv'}, $\bMV:[0,T]\times  \R^{d_Y}  \times \Omega  \rightarrow \R$ is $\bbb_T \otimes \bbb^{d_Y} \otimes \{\emptyset, \Omega' \} \otimes \cff''_T  $-measurable. Thus $\ome \mapsto \bMV(t,\bY^{\bctrl}_t(\ome); \ome)$ is $\cff'_t \otimes  \cff''_T$-measurable for any fixed $t\in[s,T]$ and \(\eta\in\caa^2\), thereby also including the case  \(\eta\in\caa^1 \subset \caa^2 \). 
By our dynamical programming principle for rough SDEs (\cref{thm-RBP}), we have the following version for the pathwise stochastic control problem. Following Remark \ref{rem:remove_running_cost} we can and will take zero running cost $\ell=0$.

\begin{theorem}[DPP for the pathwise stochastic control]
Suppose that  \cref{assum-path-stoch} holds. Then, for any $y \in \R^{d_Y}$,  $s\in[0,T]$, 
$h\in[0,T-s]$ and $i=1,2$, 
\be\label{DPP-path-stoch}
\bMV(s,y;\ome)= {\tmop{essinf}_{\eta \in \mathcal{A}^i}}\ \E[ \bMV(s+h, \bY^{s,y,\bctrl}_{s+h}; \cdot) \mid \cff^{\BW}_T] (\ome) \quad\text{a.e. } \ome,
\ee
where the left-hand side is continuous on $[0,T]\times \R^{d_Y}$ a.s. as in Theorem \ref{rep-value-fct}.

\end{theorem}

\begin{proof}
Let $\bMV^i(s,h,y,\ome),$ $i=1,2$ be the right-hand side of \eqref{DPP-path-stoch}.
 Similar to the proof of Theorem \ref{rep-value-fct}, since $\MA^1 \subseteq \MA^2, $ we easily have $ \bMV^1 \ge \bMV^2$ a.s.. Then we only need to show that  $\bMV^2 \ge \bMV$ and $\bMV \ge \bMV^1$  a.s.. 
 
 For ``$\bMV^2 \ge \bMV$'' part, take any $u \in \MA^2.$ By \cref{cost-g-X}, we have 
 \[
 \E[ \bMV(s+h, \bY^{s,y,\bar{u}}_{s+h}; \cdot) \mid {\text{$\cff$}}^{\BW}_T](\omega) = \E'[\MV(s+h, Y^{(s,y,u,\BX) } ; \BX)] \Big|_{\BX=\BW(\ome'')}.
 \]	
Recall $\MA$ is defined by \eqref{rough-ad-contr}, and thus $u(\cdot,\BX) \in \MA$ for any fixed $\BX$. By Theorem \ref{thm-RBP} (rough DPP) it follows that
\be 
\begin{split}
\E[ \bMV(s+h, \bY^{s,y,\bar{u}}_{s+h}; \cdot) \mid \cff^{\BW}_T](\ome) \ge & \left(\inf_{\ctrl \in \MA} \E'[\MV(s+h,Y^{(s,y,\ctrl,\BX)}_{s+h};\BX)] \right) \Big|_{\BX=\BW(\ome'')}	\\
=& \ \MV(s,y;\BX)\Big|_{\BX=\BW(\ome'')}= \bMV(s,y;\ome).
\end{split}
\ee

For ``$\bMV \ge \bMV^1$'' part, again by Theorem \ref{thm-RBP}, we have 
$$ 
\bMV= \left(\inf_{\ctrl \in \MA} \E'[\MV(s+h,Y^{(s,y,\ctrl,\BX)}_{s+h};\BX)] \right) \Big|_{\BX=\BW(\ome'')}. 
$$ 
We claim that for any $\vep >0,$ 
\[
\left(\inf_{\ctrl \in \MA} \E'[\MV(s+h,Y^{(s,y,\ctrl,\BX)}_{s+h};\BX)] \right) \Big|_{\BX=\BW(\ome'')} \ge \bMV^1(\ome) - 3 \vep,
\]
which implies our result.
Indeed, since $\MV(s+h,\cdot; \cdot)$ is continuous on $\R^{d_Y} \times \Coa_T$ in view of  \cref{thm:RoughValueReg}, we can repeat the second part of the proof of Theorem \ref{rep-value-fct} with $g(Y^{\ctrl,\BX}_T)$ replaced by $\MV(s+h,Y^{s,y,\ctrl,\BX}_{s+h};\BX),$ and thus the claim follows.
\end{proof}

\section{Pathwise stochastic control with Brownian rough path noise } \label{sec:BMcase}


Now in particular, we take $(\Omega'',\mathcal G'',(\cff''_t)_t,\P'')$ as any filtered probability space with a $d_W$-dimensional Brownian motion $W.$ Again consider
\begin{equation}  \label{secBMproductbasis}
	 	( {\Omega},\MG,(\MF_t)_t,\P)  = (\Omega',\mathcal G',(\cff'_t)_t,\P') \otimes (\Omega'',\mathcal G'',(\cff''_t)_t,\P''),
\end{equation} 
where $(\Omega',\mathcal G',(\cff'_t)_t,\P')$ is defined as in Section \ref{sec-randomised-rsde}. 

This section is devoted to the case when $\BX$ happens to be some realisation of some Brownian rough path $\BW = \BW (\omega'')$, over some Brownian motion $W=W(\omega'')$, independent of $B = B(\omega')$,
in which case one may expect connections to ``doubly stochastic'' It\^o stochastic differential equations, provided controls are causal in $\BX$.

We review the standard construction of the It\^o and Stratonovich Brownian rough path, following Chapter 3 in \cite{FH20}.
For each $(s,t)\in \Delta_T$, define $\mathbb{W}^{\text{It\^o}}_{s,t}=\int_s^t \delta W_{s,r}dW_r$ as an It\^o integral and $\mathbb{W}_{s,t}^{\mathrm{Strato}}:=\int_s^t \delta W_{s,r}\circ dW_r=\mathbb{W}^{\text{It\^o}}_{s,t}+\frac12 (t-s)I$ as a Stratonovich integral.
It is well-understood that $\BW^{\text{It\^o}}(\omega'')=(W(\omega''),\mathbb{W}^{\text{It\^o}}(\omega''))$ and $\BW^{\mathrm{Strato}}(\omega'')=(W(\omega''),\mathbb{W}^{\mathrm{Strato}}(\omega''))$ are elements of  $\cap_{\alpha\in(1/3,1/2)} \Coa_T$ for all $\omega''\in N_1^c$, where $N_1$ is a $\P''$-null set of $\Omega''$. Moreover, for any $t\in[0,T]$, the {\it lifting mapping } 
\be\label{ito-lift}
\begin{array}{llll}
\BW:=\BW^{\text{It\^o}}: & [0,T]\times \Omega'' & \rightarrow & (\Coa_T, \BCoa_T)	\\
 & (t,\ome'') &\mapsto & \BW^{\text{It\^o}}(\ome'')_{.\wedge t}
 \end{array}
\ee 
is progressively measurable w.r.t. $(\MF''_t )_t$. 
Moreover, we have the following result on the progressive measurability for the composition of $\BW^{\text{It\^o}}$ with $\BCoa_T$-optional process introduced in Definition \ref{def:optional}.

%

\begin{prop}\label{prop:rou-inte-agr-ito} 
 Suppose that $(Z,Z')$ defined on $[0,T] \times \Omega' \times \Coa_T$ is $ \BCoa_T$-optional, and satisfies $(1)$ $(Z_0, Z'_0)(\ \cdot \ ; \BW_0) \in L_q(\MF_0;\P)$; $(2)$ for any $\BX=(X,\X) \in \Coa_T,$ $(Z,Z')^{\BX}(\cdot):= (Z,Z')(\ \cdot \ ; \BX) \in \BD^{\gamma,\gamma'}_X L_{p,q}(\Omega')$, where $q \ge p \ge 2,$ 
 \(\gamma,\gamma'\in[0,1]\) such that \(\alpha+\gamma>1/2\) and \(\alpha+\min(\alpha,\gamma)+\gamma'>1\);
 $(3)$ $(Z, Z')$ is $\BX$-{\it causal}.\
 Then for any $t\in[0,T],$ 
\[ 
\int_0^t Z^{\mathbf{W} (\ome'')}_s (\omega') d W_s (\omega'')= \Big(  \int_0^t (Z, Z')^{\mathbf{X}} (\omega') d
   \mathbf{X} \Big)   \Big|_{\mathbf{X} = \mathbf{W} (\omega'')}, \ \ \P\text{-a.s.}
\]
where the left-hand side is the It\^o integral.
 
\end{prop} 
\begin{proof}
   By the $\mathbf{X}$-causality of $(Z, Z')^{\mathbf{X}}$,
   we have that $(Z,Z')^{\BW(\cdot)}(\cdot)$ is $(\MF_s)$-adapted, and thus the It\^o integral is well-defined. 
   Let \(J=(J^\BX_t)\) be a \(\BCoa_T\)-optional version of the rough stochastic integral 
   $
   \int_0^t (Z, Z')^{\mathbf{X}} (\omega') d
      \mathbf{X}
   $ whose existence is guaranteed by \cref{optional-r.i.}.  
   From  \cref{prop.rsint}, we see that for any fixed \(t\) and \(\BX\), \(J^\BX_t\) is the limit in \(\P'\)-probability of compensated Riemann sums
   approximation, denoted by  $\{I_n^{\mathbf{X}} (\omega') : n \geqslant 1 \}$.
   By Corollary \ref{coro-ito-rough}, it also holds that
   \[ 
   \lim_{n \to \infty} I_n^{\mathbf{W} (\ome'')} (\omega') = \Big( \int_0^t
      (Z, Z')^{\mathbf{X} } (\omega') d \mathbf{X} \Big) (\omega',
      \mathbf{X}) \Big|_{ \mathbf{X} = \mathbf{W} (\omega'')}   \quad
      \text{in $(\mathbb{P}' \otimes \mathbb{P}'')$-probability. } 
   \]
   By assumption, $W = W (\omega'')$ is a $\left( {\cff''_s}
   \right)$-Brownian motion on  $({\Omega''},
   ( {\cff}''_{s})_s,
   \ensuremath{\mathbb{P}}^{''})$, thus also a $(\cff_s)$-Brownian
   motion on the product space $(\Omega, (\cff_s)_s, \mathbb{P}) = (\Omega',
   ( \cff'_s)_s, \mathbb{P}') \otimes (\Omega'', ( \cff''_s)_s,
   \mathbb{P}'')$.
   
   By an argument similar to \cite[Prop. 5.1]{FH20}, we see that the limits in $(\mathbb{P}' \otimes
   \mathbb{P}'')$-probability,
   \[ 
   \lim_{n \to \infty} I_n^{\mathbf{W} (\ome'') } (\omega') 
   \]
   and
   \[ \lim_{n \to \infty} \sum_{[s, r] \in \pi_n} Z_s^{\mathbf{W} (\ome'') }
      (\omega') (W_r - W_s) (\omega'') \]
   exist and coincide. 
   Noting that the latter is the It\^o integral, we obtain the claimed identity.
\end{proof}

To study the equivalence between the randomised rough system \eqref{randomised-rsde} and the stochastic system with two Brownian motions, we need to consider the following $\mathbf{X}$-causal  controls, in which case there is no problem in the definition of the later system.
\begin{equation}
\label{equ:Acausal}
\MA^{\mathrm{causal}}:=\{ \ctrl= \ctrl_t(\ome', {\text{$\BX$}}_{.\wedge t} ) \text{ is measurable w.r.t. }\MO'\otimes \BCoa_T  \}.
\end{equation}
We note that control processes $\bar \eta$, for $\eta \in \MA^{\mathrm{causal}}$, are adapted w.r.t $\mathfrak{F} =  \mathfrak{F}' \otimes \mathfrak{F}''$, hence correspond to the ``admissible'' controls considered in \cite{BM07}.

%

\begin{prop}
	Suppose that $(b,\sigma,f)$ satisfies Assumption \ref{assum-path-stoch}, and $\BW=\BW^{\text{It\^o}}$ is given by \eqref{ito-lift}. For any $\ctrl \in \MA^{\mathrm{causal}}$, let $\bY^{\bar \eta }$ be defined by \eqref{randomised-rsde}, and $Y$ be the solution to the SDE 
\be\label{BB-rde}
\left\{
\begin{aligned}
 & dY_t =b(t,Y_t ,\bctrl_t   )dt+\sigma(t,Y_t ,\bctrl_t )dB_t(\omega')+f(t,Y_t ) dW_t(\ome''), \\
& Y_s = y. 
 \end{aligned}\right.
\ee
with $\bctrl_t(\ome',\ome''):= \ctrl_t(\ome',\BW(\ome''))$.
Then for any 
$t\in[0,T], $ we have $\bY^{\bar \eta }_t= Y_t$, $\P$-a.s.. 	
	
\end{prop}


\begin{proof}
   We recall that \(Y^{\eta,\BX}\) is the \(\BCoa_T\)-optional solution to \eqref{r-dsde'}.
   Moreover, because \(\eta\in\caa^{\mathrm{causal}} \), it is evident that $\P'$-a.s. $Y^{\ctrl, \BX}_t = Y^{\ctrl, \BX_{\cdot \wedge t}}_t$ for all $t \in [0,T].$ 
   By \cref{def:solutionRSDEs_compendium}, we see that \((Z,Z'):=(f(Y^{\eta,\BX} ), D_yf(Y^{\eta,\BX} )f(Y^{\eta,\BX} ))\) belongs to \(\textbf{D}^{\bar \alpha,\bar \alpha'} _XL_2\) for some suitable parameters \(\bar \alpha,\bar{\alpha}' \) stated there. 
   Then according to Proposition \ref{prop:rou-inte-agr-ito}, 
\be\label{eq-rs1}
\int_s^t f_r( Y^{\ctrl, \BX}_r ) d\BX_r \Big|_{\BX= \BW(\ome'')}= \int_s^t f_r (\bY^{\bar \eta}_r) dW_r(\ome'').
\ee
On the other hand, we note that for any $(t,\BX) \in [s,T] \times \Coa_T,$ let 
$$
A_{i,n}^t(\ome', \BX) := \frac{i}{2^n} \text{Leb}\left( \left\{r\in[s,t]: b_r(Y^{\ctrl, \BX }_r , \ctrl_r( \BX) ) \in [ \frac{i}{2^n},  \frac{i+1}{2^n}) \right\}  \right),$$ 
where $\ctrl_r(\BX)=(\ome'\mapsto \ctrl_r(\ome',\BX)) $. Then we have
$$
\int_0^t b_r(Y^{\ctrl, \BX }_r , \ctrl_r( \BX) )dr 
=:\lim_{n \rightarrow \infty} \sum_{i=-n2^n}^{n2^n-1}  A_{i,n}^t(\ome', \BX),
$$
for each $\omega'$ and hence in $\mathbb{P}'$-probability.
By Corollary \ref{coro-ito-rough}, we have the following convergence in $\P$-probability
\be\label{eq-rs2}
\begin{split}
  \int_s^t b_r(Y^{\ctrl, \BX }_r , \ctrl_r(\BX) )dr \Big|_{\BX=\BW(\ome'')}& =  \lim_{n \rightarrow \infty} \sum_{i=-n2^n}^{n2^n-1}  A_{i,n}^t(\ome', \BX) \Big|_{\BX=\BW(\ome'')} ,	\\
 & = \lim_{n \rightarrow \infty} \sum_{i=-n2^n}^{n2^n-1}  \bar{A}_{i,n}^t,\\
  & = \int_s^t b_r( \bY^{\bctrl }_r, \bctrl_r  ) dr,
\end{split}
\ee
where $\bar{A}_{i,n}^t:= \frac{i}{2^n} \text{Leb}\left(  \{r\in[s,t]: b_r( \bar{Y}^{\bctrl }_r , \bctrl_r  ) \in [ \frac{i}{2^n},  \frac{i+1}{2^n})  \}  \right)$. 
Corollary \ref{coro-ito-rough} also leads to
\be\label{eq-rs3}
\int_s^t \sigma_r(Y^{\ctrl, \BX }_r , \ctrl_r(\BX) )dB_r \Big|_{\BX=\BW(\ome'')} = \int_s^t \sigma_r( \bY^{\bctrl }_r, \bctrl_r  ) dB_r.
\ee
In view of \eqref{eq-rs1}, \eqref{eq-rs2} and \eqref{eq-rs3}, we obtain 
\be
\bY^{\bctrl}_t = y + \int_s^t b_r( \bY^{\bctrl }_r, \bctrl_r  ) dr + \int_s^t \sigma_r( \bY^{\bctrl }_r, \bctrl_r  ) dB_r + \int_s^t f_r (\bY^{\bar \eta}_r) dW_r(\ome''),
\ee
which implies that $\bY^{\bctrl}$ is a solution to \eqref{BB-rde}, and our result follows by the uniqueness of solutions to \eqref{BB-rde}.
\end{proof}


\appendix
\section*{Some measurable selection results}

We collect some facts on measurable selection, mostly from Stricker--Yor {\cite{SY78}}. 

\medskip 
\noindent Let 
$(\Omega, \cff, (\cff_t)_{t \geq 0}, P)$ is a filtered
probability space satisfying the usual conditions.
We agree that
$\cff_{0_-} =\cff_0$, and that $\cff=\cff_{\infty}
= \bigvee_t \cff_t$. We denote by $\ooo$ (resp. $\mathcal{P}$)
the optional (resp. predictable) sigma-algebra on $\mathbb{R}_+ \times \Omega$
associated with this family.

Assume a given measurable space $(U, \uuu)$, and we consider functions
$X : (u, t, \omega) \mapsto X^u_t (\omega)$ with real values on $U \times
\mathbb{R}_+ \times \Omega$. We will say that $X$ is $\uuu$-{\em{measurable}}\footnote{Stricker--Yor call such $X$ simply measurable, we prefer to be more explicit.}
without further clarification to indicate that $X$ is measurable with respect
to $\uuu \otimes \mathcal{B} (\mathbb{R}_+) \otimes \cff$. We
always interpret $X$ as a family, indexed by $u$, of stochastic processes $X^u
= (X^u_t)_{t \geq 0}$, and we will consider the phrase ``the process $X^u$
depends measurably on $u$'' as equivalent to ``$X$ is $\uuu$-measurable''.

Let $X$ and $\bar{X}$ be two functions as described above; we will sometimes
say that $X$ and $\bar{X}$ are {\tmem{indistinguishable}} if, for all $u \in
U$, the processes $X^u$ and $\bar{X}^u$ are indistinguishable.

\begin{lem}[\cite{SY78}, Lemme 1] Let $X$ be a function on $U \times \mathbb{R}_+
  \times \Omega$ such that:
  \begin{itemize}
    \item[a)] For all $u \in U$, the process $X^u$ is indistinguishable from a
    c{\`a}dl{\`a}g process.
    
    \item[b)] For all $t \in \mathbb{R}_+$, there exists a function $H_t$ on
    $U \times \Omega$, measurable with respect to $\uuu \otimes
    \cff$, such that: $\forall u \in U, X^u_t (\cdummy) = H_t (u,
    \cdot)$ almost surely $P$-a.s.
  \end{itemize}
  Then there exists a $\uuu$-measurable function $Y$ on $U \times \mathbb{R}_+ \times
  \Omega$, indistinguishable from $X$, such that for all $u$, the trajectories
  of $Y^u$ are all c{\`a}dl{\`a}g.
\end{lem}

Suppose that $H_t$ is $\cff_t$-measurable for all $t$. The family
$(\cff_t)$ satisfies the usual conditions, the $P$-negligible set $C_u$
belongs to $\cff_0$, and the process $(Y^u_t)$ is adapted to
$(\cff_t)$ for all $t$ and is c{\`a}dl{\`a}g, hence optional, for all
$u \in U$. However, as pointed out in \cite{SY78} (after Lemme 1), the map $(u, (t, \omega)) \mapsto
Y^u_t (\omega)$ is not $\uuu \otimes \ooo$-measurable a priori
(see however Proposition \ref{L3SY78} and the remark which follows).

%

\begin{lem}\label{SY78-Lem2} 
  Let $H : (u, t, \omega) \mapsto H^u_t (\omega)$ be a $\uuu \otimes \ooo$-measurable
  function such that, for all $u$ and all finite $t$, we have
  \[ \int_0^t |H^u_s (\omega) |ds < \infty \quad P
     \text{-a.s.} \]
  Then there exists  $Z$ measurable w.r.t. $\uuu \otimes \ooo$ such that, for all $u$, the
  processes $(Z^u_t)$ and $\left( \int_0^t H^u_s d s \right)$ are
  indistinguishable.
\end{lem}
\begin{remark} This lemma is a variation of Lemme 2 in  {\cite{SY78}}, that deals with the integration against general finite variation process.
\end{remark} 
\begin{proof}
For any fixed $t,$ by the joint measurability of $H$ on $(U \times [0,t] \times \Omega , \uuu \otimes \bbb_t \otimes \MF_t) $, we see that $ \int_0^t H^u_s d s $ is $\uuu \otimes \MF_t$-measurable. Note that $\int_0^t H^u_s d s$ is continuous in $t$, and our conclusion follows. 
\end{proof}

\begin{proposition}[{\cite{SY78}}, Proposition 1 and Corollary 1] \label{SY_Prop1}
 (i) Let $(X_n)$ be a sequence of
  $\uuu \otimes \cff$-measurable functions on $U \times \Omega$.
  Suppose that for all $u \in U$, the sequence $X_n (u, \cdot)$ converges in
  probability on $\Omega$. Then there exists a $\uuu \otimes
  \cff$-measurable function $X$ such that for all $u \in U$, $X (u,
  \cdot) = \lim_{n \to \infty} X_n (u, \cdot)$ in probability.
  
  (ii) Suppose that $X_n (u, \cdot)$ converges in $L^p$ for all $u \in U$.
  Then there exists a $\uuu$-measurable function $X$ on $U \times \Omega$ such that,
  for all $u$, $X_n (u, \cdot)$ converges in $L^p$ to $X (u, \cdot)$. 
\end{proposition}


\begin{corollary}\label{coro-ito-rough}
  Let $Q$ be a probability measure on the parameter space $(U, \uuu)$.
  For $X$ and $X_n$ as in Proposition \ref{SY_Prop1}, part (i) and (ii), respectively.
  Then
  
  (i) \ $X = \lim_{n \to \infty} X_n$ in $(Q \otimes P)$-probability.
  
  (ii) $X = \lim_{n \to \infty} X_n$ in $L^p (Q \otimes P)$, provided $\{
  \mathbb{E}^P (| X_n (u, \cdot) |^p) : n \geqslant 1 \}$ is uniformly
  $Q$-integrable. 
\end{corollary}

\begin{proof}
  (i) We know from Proposition \ref{SY_Prop1}, part (i), that all $u \in U$
  \[ \phi_n (u) : =\mathbb{E}^P (| X (u, \cdot) - X_n (u, \cdot) | \wedge 1)
     \rightarrow 0. \]
  By  ($\uuu \otimes \cff$) -measurability of $X$, the Fubini Theorem and
  bounded convergence
  \[ \mathbb{E}^{(Q \otimes P)} (| X - X_n | \wedge 1) =\mathbb{E}^Q \phi_n
     \rightarrow 0 \]
  which shows that $X_n$ converges to $X$ in $(Q \otimes P)$-probability. \
  
  (ii) We know $\phi^p_n (u) : =\mathbb{E}^P (| X (u, \cdot) - X_n (u, \cdot)
  |^p) \rightarrow 0$. Assuming $\{ \mathbb{E}^P (| X_n (u, \cdot) |^p) : n
  \geqslant 1 \}$ to be uniformly $Q$-integrable, implies that the $\{
  \phi^p_n : n \geqslant 1 \}$ are UI w.r.t $Q$ and so
  \[ \mathbb{E}^{(Q \otimes P)} (| X - X_n |^p) =\mathbb{E}^Q \phi_n
     \rightarrow 0. \]
\end{proof}

For a measurable process $H : \mathbb{R}_+ \times \Omega \to \mathbb{R}$
(which we assume, for simplicity, to be bounded or positive), and let
$^{\circ} H$ and $^p H$ be its optional and predictable projections. Recall
that $^{\circ} H$, for example, is closely related to certain conditional
expectation operators:
\[ \text{-- If } T \text{ is an } (\cff_t)  \text{-stopping time,
   }^{\circ} H_T =\mathbb{E} [H_T |\cff_T],\ \ P\text{-a.s. on } \{T <
   \infty\} . \]
\begin{proposition}[{\cite{SY78}}, Lemme 3]  \label{L3SY78}Let $H : U \times \mathbb{R}_+ \times \Omega \to
  \mathbb{R}$ be a measurable function, bounded or positive. There exist two
  functions $K$ and $L : U \times (\mathbb{R}_+ \times \Omega) \to
  \mathbb{R}$, measurable respectively with respect to $\uuu \otimes
  \ooo$ and $\uuu \otimes \mathcal{P}$, such that for all $u \in
  U$, $K^u$ (resp. $L^u$) is a version of $^{\circ} H^u$ (resp. $^p H^u$). We
  will write $K = {^{\circ} H}$, $L = {^p H}$ from now.
\end{proposition}

%

\begin{rem} \label{rem_optional}Suppose that a measurable function $H$ on $U \times
\mathbb{R}_+ \times \Omega$ is such that, for all $u$, the process $(H^u_t)$
is optional. Then there exists a function $K$ indistinguishable from $H$ such
that $(u, (t, \omega)) \mapsto K (u, t, \omega)$ is $\uuu \otimes
\ooo$-measurable. Indeed, one immediately reduces this by means of a
bijection from $\mathbb{R}$ onto $] 0, 1 [$ in the case where $H$ is bounded,
and it is then sufficient to take $K = {^{\circ} H}$. 
\end{rem}

\begin{prop}[{\cite{SY78}}, Proposition 5]\label{SY78-Prop5}
   Let $M$ be a local martingale, and $J : U \times
  \mathbb{R}_+ \times \Omega \to \mathbb{R}$ a $\uuu \otimes
  \ooo$-measurable function such that
  \[ 
  \forall u \in U, \quad \mathbb{E} \left[ \left( \int_0^{\infty} (J^u_s)^2
     d [M, M]_s \right)^{\frac12} \right] < \infty . 
  \]
  Then there exists a $\uuu \otimes \ooo$-measurable function $Y
  : U \times \mathbb{R}_+ \times \Omega \to \mathbb{R}$ such that, for all $u
  \in U$, the process $(Y_t^u)$ is indistinguishable from the optional
  stochastic integral $\left( \int_0^t J^u_s dM_s \right)_{t \geq 0}$.
\end{prop}

\begin{remark}

If $M$ is only a semimartingale, similar results hold but $J$ needs to be $\uuu$-predictable ($\uuu \otimes \mathcal{P}$-measurable), cf. Theorem 1 in {\cite{SY78}}.

\end{remark}

\section*{Acknowledgments}
The authors would like to thank the anonymous referees, an Associate
Editor and the Editor for their constructive comments that improved the
quality of this paper.

\section*{Funding}
PKF and HZ acknowledge support from DFG CRC/TRR 388 ``Rough
Analysis, Stochastic Dynamics and Related Fields'', Projects A07, B04 and B05. Part of this work was carried out during a visit of the first author to Shandong University. 
KL acknowledges supports from EPSRC
[grant number EP/Y016955/1] and from a Humboldt Research Fellowship, while at TU Berlin where this project was commenced. HZ is partially supported by Fundamental Research Funds for the Central Universities, NSF of Shandong Province (Grant Numbers ZR2023MA026, ZR2023ZD35).

\bibliography{processes_martingale}

\end{document}